\documentclass[reqno,12pt]{amsart}

\usepackage[margin=1in]{geometry}
\usepackage{amsmath,amssymb,amsthm}
\usepackage[hidelinks]{hyperref}
\usepackage{microtype}

	\newcommand{\eps}{\varepsilon}
\newcommand{\R}{\mathbb R}

	\newcounter{dummy} \numberwithin{dummy}{section}

\newtheorem{corollary}[dummy]{Corollary}
\newtheorem{theorem}[dummy]{Theorem}
\newtheorem{lemma}[dummy]{Lemma}

\newtheorem{problem}[dummy]{Problem}

\title[Section functions]{On determination of convex bodies by their section functions}

\author{Alexander Fish, Dmitry Ryabogin, Wen Rui Sun, and Vladyslav Yaskin}

\address{School of Mathematics and Statistics, University of Sydney,  NSW 2006, Australia}
\email{alexander.fish@sydney.edu.au}

\address{Department of Mathematical Sciences, Kent State University, Kent, OH 44242,
USA}
\email{ryabogin@math.kent.edu}

\address{Department of Mathematical and Statistical Sciences, University of Alberta, Edmonton, Alberta, T6G 2G1, Canada}
\email{wrsun@ualberta.ca}
\address{Department of Mathematical and Statistical Sciences, University of Alberta, Edmonton, Alberta, T6G 2G1, Canada}
\email{yaskin@ualberta.ca}

\subjclass[2020]{Primary 52A20, 52A38.}
\keywords{Convex body, section, Fourier transform, spherical harmonic}
\date{}

\thanks{The first author was supported by the ARC DP240100472 Grant. The second author was supported by the US NSF Grants DMS-2247771 and DMS-2604412. The third and fourth authors were supported by an NSERC Discovery Grant.}

\begin{document}

\begin{abstract} For a convex body $K$ its parallel section function in the direction $\theta\in S^{n-1}$ is defined by 
	 $$A_{K,\theta}(t) = \mathrm{vol}_{n-1} \left( K \cap \{ x \in \mathbb{R}^n ~ | ~ \langle x, \theta \rangle = t \} \right) .$$
	 We study to what extent the body $K$ is determined by partial information about these functions. As applications, we characterize convex bodies with locally separable
	 section functions, establish partial results on the homothety
	 conjecture for bodies of flotation, and answer a question of Barker
	 and Larman concerning the determination of convex bodies from section
	 functions at infinitely many distances from the origin.

\end{abstract}

\maketitle

\section{Introduction}

Geometric tomography is an area of mathematics that deals with the retrieval of information
about solids from lower-dimensional data, such as the areas of sections or projections.
 The book ``Geometric Tomography" by R.~J.~Gardner \cite{g} gives a comprehensive account of known results, techniques and open problems in the area. 

Throughout the paper, \(S^{n-1}\) denotes the unit sphere in
\(\mathbb R^n\), equipped with the standard inner product
\(\langle\cdot,\cdot\rangle\), and
\[
\xi^\perp=\{x\in\mathbb R^n:\langle x,\xi\rangle=0\}.
\]
We denote by \(B_2^n\) the Euclidean unit ball in \(\mathbb R^n\).
A {\it convex body} is a compact convex subset of \(\mathbb R^n\) with non-empty
interior.

For a convex body \(K\subset\mathbb R^n\), its {\it parallel section function}
in the direction \(\theta\in S^{n-1}\) is defined by
\[
A_{K,\theta}(t)
=
\operatorname{vol}_{n-1}
\bigl(
K\cap\{x\in\mathbb R^n:\langle x,\theta\rangle=t\}
\bigr).
\]
A basic question in geometric tomography is to determine how much
information about \(K\) is encoded by these functions.

A classical result of Minkowski and Funk states that if \(K\) and \(L\)
are origin-symmetric convex bodies in \(\mathbb R^n\) such that
\[
\operatorname{vol}_{n-1}(K\cap\xi^\perp)
=
\operatorname{vol}_{n-1}(L\cap\xi^\perp)
\]
for every \(\xi\in S^{n-1}\), then \(K=L\); see
\cite[Corollary~7.2.7]{g}. Thus an origin-symmetric convex body is
uniquely determined by the values of 
\(
A_{K,\theta}(0),\) \( \theta\in S^{n-1}.
\)

For non-central sections much less is known. The following problem, posed
by Barker and Larman \cite{BL}, remains open even in the plane.

\begin{problem}
	Let \(K\) and \(L\) be convex bodies in \(\mathbb R^n\) such that
	\[
	tB_2^n\subset \operatorname{int}K\cap\operatorname{int}L
	\]
	for some \(t>0\). Suppose that
\begin{equation}\label{noncentral}
	\mathrm{vol}_{n-1}(K\cap H) = \mathrm{vol}_{n-1}(L\cap H)
\end{equation}
	for every hyperplane \(H\) tangent to \(tB_2^n\). Is then \(K=L\)?
\end{problem}

Note that condition \eqref{noncentral} is equivalent to 
\[
A_{K,\theta}(t)=A_{L,\theta}(t),
\qquad \theta\in S^{n-1}.
\]
An analogue of this problem on the sphere was considered earlier by
Santal\'o \cite{S}, and some partial results were obtained in \cite{H}, \cite{HR}, \cite{XMC}, \cite{Y}, \cite{YZh}.
Since the problem remains open for a single value of \(t\), it
is natural to ask what happens if equality of the section functions is
known for all \(t\) in some set $A$.

Barker and Larman proved that in odd dimensions equality of the parallel
section functions for infinitely many values of \(t\) implies \(K=L\).
They also asked whether the corresponding result holds in even
dimensions. In Section~6 we answer this question affirmatively.  

A key ingredient in this and  other results of the paper is an
analyticity property of the spherical harmonic components of the
parallel section function.  Namely, if $K$ is a convex body containing the origin in its interior, then for every spherical harmonic \(H\),
the function
\[
t\longmapsto
\int_{S^{n-1}}A_{K,\theta}(t)H(\theta)\,d\theta
\]
is real analytic on  $(-a,a)$, where \(a=\min_{x\in \partial K}|x|\).   When \(n\) is odd, this function
is in fact a polynomial on this interval. 

Another direction of research concerns characterizations of ellipsoids
in terms of their parallel section functions. Motivated by Arnold's
problem on algebraically integrable bodies (see Problems 1987-14, 1988-13, and 1990-27 in \cite{Ar}), Agranovsky \cite{A} introduced the following notion. A convex body \(K\subset\mathbb R^n\) is called
{\it polynomially integrable} if, for every \(\theta\in S^{n-1}\),
\[
A_{K,\theta}(t)
=
\sum_{m=0}^N a_m(\theta)t^m
\]
for some non-negative integer \(N\), some functions $a_m$ on the sphere,  and for all \(t\) for which the
section
\(
K\cap\{x\in\mathbb R^n:\langle x,\theta\rangle=t\}
\)
is non-empty.

Agranovsky proved that infinitely smooth polynomially integrable convex bodies do not exist in
even dimensions and conjectured that in odd dimensions they should be
ellipsoids. The latter conjecture was confirmed in
\cite{kmy}. In fact, a local version is sufficient: if
$$A_{K,\theta} (t)=\sum_{m=0}^N a_m(\theta) t^m,$$ for  all $\theta\in S^{n-1}$ and all $t$ from some small interval containing the origin, then $K$ is an ellipsoid in odd dimensions.

A natural next step is to consider convex bodies whose parallel section
functions admit a  representation of the form
\begin{equation}\label{def:separable}
A_{K,\theta}(t)
=
\sum_{m=0}^{N} a_m(\theta)b_m(t).
\end{equation}
Let $K\subset \R^n$ be a convex body containing the origin in its interior.
We say that \(K\) has a {\it locally separable section function near the
origin} if there exists \(\varepsilon>0\) such that
\eqref{def:separable} holds for every \(\theta\in S^{n-1}\) and every
\(t\in(-\varepsilon,\varepsilon)\). We say that \(K\) has a {\it locally separable section function near the boundary} if there
exists \(\varepsilon>0\) such that \eqref{def:separable} holds for every
\(\theta\in S^{n-1}\) and every
\[
t\in\bigl(h_K(\theta)-\varepsilon,h_K(\theta)\bigr),
\]
where $h_K$ is the support function of $K$.

These concepts are closely related to the notion of separably integrable bodies introduced by Yaskin and Zawalski \cite{yz}. 
Following  \cite{yz}, we say that a convex body $K$ is {\it locally 1-separably integrable} if $$V_{K,\theta}(t) = \int_{-t}^t A_{K,\theta} (s) \, ds =\sum_{m=0}^N a_m(\theta) b_m(t),$$ 
for all $\theta\in S^{n-1}$ and all $t$ in a small interval containing zero.

Let $K$ be an origin-symmetric convex body in $\mathbb R^n$ (i.e., $K=-K$) with infinitely smooth boundary.  It was shown in \cite{yz} that if $K$ is locally 1-separably integrable, then $K$ is an ellipsoid if $n$ is odd, and $K$ is a Euclidean ball if $n$ is even.

Note that a body with a locally separable section function near the origin is necessarily locally 1-separably integrable. For origin-symmetric convex bodies, these two notions coincide.

 In Section \ref{Section-polynomial-separable}, we  show that if a  convex body $K \subset \mathbb R^n$ (not necessarily origin-symmetric or smooth) has a locally separable section function near the origin, then $K$ is an ellipsoid if $n$ is odd, and $K$ is a centered Euclidean ball if $n$ is even. We will also obtain the same conclusion for infinitely smooth convex bodies in even dimensions with locally separable section functions near the boundary.
Our results do not cover the case when $n$ is odd and the body has a separable section function near the boundary. In fact, a version of this question for polynomially integrable bodies, posed by Agranovsky in \cite{A2}, is also open.

We finally discuss the homothety conjecture. Let $K$ be a convex body in $\mathbb R^n$ and let $\delta$ be a real number in $(0,\mathrm{vol}_n(K)/2)$. For each $\theta\in S^{n-1}$ and $t\in \mathbb R$, define
$$W^+(\theta,t)= \{x\in\mathbb R^n: \langle x,\theta\rangle\ge t\},\qquad \mathrm{and} \qquad W^-(\theta,t)= \{x\in\mathbb R^n: \langle x,\theta\rangle\le t\}.$$ 
For each $\theta$ there is a unique $t(\theta)$ such that 
$$\mathrm{vol}_n(K\cap W^+(\theta,t(\theta)))=\delta.$$
The {convex body of flotation} (also known as the convex floating body), is defined by 
$$K_\delta = \bigcap_{\theta\in S^{n-1}} W^-(\theta,t(\theta)).$$
This concept was introduced in the works of B\'ar\'any and Larman \cite{BarL}, and Sch\"utt and Werner \cite{SW}. 

The homothety conjecture, proposed by Sch\"utt and Werner \cite{SW2}, says that if a convex body in $\mathbb R^n$ is homothetic to one  of its convex bodies of flotation, then it is an ellipsoid. For partial results, the reader is referred to  \cite{ANRSY}, \cite{St}, \cite{WY}.

We consider the following version of this problem. Let $K$ be a  convex body  in $\mathbb R^n$ that is homothetic to $K_\delta$ for each $\delta$ from some interval $(\delta_1,\delta_2)\subset (0, \mathrm{vol}_n(K)/2)$. Does this imply that $K$ is an ellipsoid?  In Section \ref{Section-floating}, we  show that if $K\subset \mathbb R^2$ is a convex body containing the origin in its interior such that	$K_\delta=\lambda(\delta)K$ for every $\delta\in (\delta_1,\delta_2)$, then $K$ is an ellipse.  Note that for a single $\delta$ this result is not true, namely,  there exists a non-symmetric convex body  $K\subset \mathbb R^2$  and $\delta>0$ such that  $K_\delta=\lambda(\delta)K$; see \cite{ANRSY}. 
In dimensions $n\ge 3$, we show that an origin-symmetric infinitely smooth convex body $K$ is an ellipsoid under the  assumption that $K_\delta=\lambda(\delta)K$ for all  $\delta\in (\delta_1, \mathrm{vol}_n(K)/2)$.

\section{Preliminaries}

 We say that a compact set $K \subset \mathbb{R}^n$ is a {\it star body} if for every point $x \in K$, each point of the interval $[0, x)$ lies in the interior of $K$, and the {\it Minkowski functional} of $K$,  defined by
 $$
 \|x\|_K = \min\{\lambda \geq 0:  x \in \lambda K \},
 $$
  is a continuous function on $\mathbb{R}^n$.
 
 The {\it radial function} of a star body $K$ is given by
 $$
 \rho_K(\theta) = \max \{\lambda > 0: \lambda \theta \in K \}, \quad \theta \in S^{n-1}.
 $$
 It is clear that $\rho_K(\theta) = \|\theta\|_K^{-1}$ for any $\theta \in S^{n-1}$, and $\rho_K$ is positive and continuous on $S^{n-1}$. Geometrically, $\rho_K(\theta)$ is the distance from the origin to the boundary of $K$ in the direction $\theta\in S^{n-1}$.

 Let $K\subset \mathbb R^n$ be a convex body.
The {\it support function} of $K$ is 
\begin{align*}
    h_K(x) = \max_{y \in K} \langle y, x \rangle , \qquad x \in \mathbb{R}^n.
\end{align*}
If $K$ contains the origin in its interior, then $h_K$ is the Minkowski functional of $K^\circ$, the polar body of $K$:
\begin{align*}
     h_K(x) =  \left\|x\right\|_{K^\circ},  \qquad x \in \mathbb{R}^n.
\end{align*}

We say that a convex body \(K\subset\mathbb R^n\) is
\emph{infinitely smooth} if its boundary \(\partial K\) is a
\(C^\infty\) hypersurface. If $K$ contains the origin it its interior, this  is
equivalent to
\(
\|\cdot\|_K\in C^\infty(S^{n-1})
\).
 In particular, if $K$ is infinitely smooth, then $A_{K,\theta}(t)$ is infinitely differentiable in $t$ on the interval $(-h_K(-\theta),h_K(\theta))$, for each $\theta\in S^{n-1}$; see \cite[Lemma 2.4]{K-book}.

A spherical harmonic of degree $m$ in dimension $n$ is defined as the restriction to the unit sphere $S^{n-1}$ of a homogeneous harmonic polynomial of degree $m$ in $n$ variables. We denote by $\mathcal{H}_m^n$ the finite-dimensional space of all spherical harmonics of degree $m \ge 0$ on $S^{n-1}$.
Spherical harmonics of different degrees are orthogonal, and every function $f\in L^2(S^{n-1})$ has a unique expansion 
$$f=\sum_{m=0}^\infty H_m,$$
where $H_m\in \mathcal{H}_m^n$, and the series is convergent in $L^2(S^{n-1})$.

We will often make use of the following Funk-Hecke formula; see \cite[Theorem 3.4.1]{Gro}.    If $f$ is a function of a single variable, such that $f(z)(1 - z^2)^{\frac{n-3}{2}}$ is integrable on $[-1,1]$, and $H \in \mathcal{H}_m^n$, then  
    \begin{align}\label{FH}
        \int_{S^{n-1}} f( \langle u ,v  \rangle) H(v)dv = \lambda_m H(u), \qquad u\in S^{n-1},
    \end{align}
    with  
    \begin{align*}
        \lambda_m=\left | S^{n-2}\right| \int_{-1}^1  f(z) P_m^n(z) (1 - z^2)^{\frac{n-3}{2}} dz,
    \end{align*}
    where $\left | S^{n-2}\right|$ is  the surface area of the sphere $S^{n-2}$, and  $P_m^n $ is the Legendre polynomial of dimension $n$ and degree $m$. It is known (see \cite[Proposition 3.3.7]{Gro}) that $P_m^n $ contains only odd powers if $m$ is odd, and only even powers if $m$ is even.

    In \eqref{FH} and throughout the paper, integration over the sphere is with respect to the rotation-invariant surface measure on $S^{n-1}$.
 
Let  $\mathcal{S}(\mathbb{R}^n)$ be the Schwartz space of complex-valued infinitely differentiable
 rapidly decreasing functions on $\mathbb{R}^n$. Elements of \(\mathcal{S}(\mathbb{R}^n)\) will be
referred to as \emph{test functions}. The space
\(\mathcal{S}'(\mathbb{R}^n)\) of \emph{tempered distributions} consists
of all continuous linear functionals on \(\mathcal{S}(\mathbb{R}^n)\).
We write
\[
\langle f,\varphi\rangle
\]
for the action of \(f\in\mathcal{S}'(\mathbb{R}^n)\) on
\(\varphi\in\mathcal{S}(\mathbb{R}^n)\).

For \(\varphi\in\mathcal{S}(\mathbb{R}^n)\), its Fourier transform is
defined by
\[
\widehat{\varphi}(x)
=
\int_{\mathbb{R}^n}
\varphi(y)e^{-i\langle x,y\rangle}\,dy.
\]
It is well known that
\(\widehat{\varphi}\in\mathcal{S}(\mathbb{R}^n)\).
 
 The Fourier transform of a  distribution $f\in  \mathcal S'(\mathbb R^n)$ is the distribution  \(\widehat f\), also denoted by \(f^\wedge\),
defined by
\begin{equation}\label{action}
	\langle \widehat{f}, \varphi\rangle = \langle {f}, \widehat{\varphi}\rangle,
\end{equation}
 for every $\varphi \in \mathcal S(\mathbb R^n)$.

For a distribution \(f\), we denote by \(f(-x)\) its reflection,
defined by
\[
\langle f(-x),\varphi(x)\rangle
=
\langle f(x),\varphi(-x)\rangle
\]
for every test function \(\varphi\).
A distribution \(f\) is called \textit{even} if \(f(x)=f(-x)\).

With our convention for the Fourier transform, every tempered
distribution \(f\) satisfies
\[
\widehat{\widehat f\,\,}(x)=(2\pi)^n f(-x).
\]
In particular, if \(f\) is even, then
\[
\widehat{\widehat f\,\,}=(2\pi)^n f.
\]

 Let \(f\in C^\infty(S^{n-1})\) and \(p\in\mathbb{C}\). Consider the
homogeneous extension of \(f\) to
\(\mathbb{R}^n\setminus\{0\}\) of degree \(-n+p\), defined by
\[
f_p(x)
=
|x|^{-n+p}
f\left(\frac{x}{|x|}\right),
\]
where $|\cdot|$ stands for the Euclidean norm on $\mathbb R^n$.

If \(\Re p>0\), then \(f_p\) is locally integrable and defines a 
distribution by
\[
\langle f_p,\varphi\rangle
=
\int_{\mathbb{R}^n}f_p(x)\varphi(x)\,dx.
\]
The Fourier transform \(\widehat{f_p}\), understood in the sense of
distributions, agrees on \(\mathbb{R}^n\setminus\{0\}\) with a
\(C^\infty\) homogeneous function of degree \(-p\); see
\cite[Lemma~3.16]{K-book} or \cite{GYY}.

 If $\Re p<0$, then $f_p$ no longer acts on test functions by integration, however one can define $\langle f_p,\varphi\rangle$ using a regularization procedure.  If \(\Re p<1\), $p$ is not an integer, and \(f\) is
 continuous, then the restriction of the Fourier transform of \(f_p\) to
\(S^{n-1}\) is given by
	\begin{align}\label{eq:GYY}
	\widehat f_p(\theta)
	&=
	\Gamma(p)\cos\frac{\pi p}{2}
	\int_{S^{n-1}}
	|\langle\theta,u\rangle|^{-p} f(u)\,du
	\notag\\
	&\quad-
	i\Gamma(p)\sin\frac{ \pi p }{2}
	\int_{S^{n-1}}
	\operatorname{sgn}\langle\theta,u\rangle
	|\langle\theta,u\rangle|^{-p} f(u)\,du, \qquad \theta\in S^{n-1};
\end{align}
 this follows from \cite[equation~(3.3)]{GYY}  by analytic
continuation. Note that even though \cite[equation~(3.3)]{GYY} is stated for infinitely smooth functions, it immediately extends to continuous functions by approximation.

If $f$ is even and continuous and $p$ is a negative odd integer, then
 \begin{equation}\label{p-odd-negative}
	\widehat{f_p} (\theta) = (-1)^{(p-1)/2} \frac{\pi}{2(-p)!}\int_{S^{n-1}} |\langle \theta, u\rangle|^{-p} f(u)\, du, \qquad \theta\in S^{n-1};
\end{equation} 
 see   \cite[Corollary 3.15]{K-book} or \cite[Equation (2.16)]{GYY}.

A connection between the derivatives of the parallel section function
of an origin-symmetric convex body and the Fourier transforms of powers
of its Minkowski functional was established in \cite{GKS}; see also \cite{K-book}.
The
corresponding formula for non-symmetric bodies was obtained in
\cite{RY}.  Let $K$ be an infinitely smooth convex body in $\R^n$ containing the origin in its interior. Let $\Re q>-1$, $q\ne n-1$, then
\begin{align}\label{FracDer-RY}
	A_{K,\theta}^{(q)}(0) &= \frac{\cos(q\pi/2) }{2\pi(n-1-q)}\Bigl(\|x\|_K^{-n+q+1}+\|-x\|_K^{-n+q+1}\Bigr)^{\wedge}(\theta)\notag\\
	&-\frac{i\sin(q\pi/2) }{2\pi(n-1-q)}\Bigl(\|x\|_K^{-n+q+1}-\|-x\|_K^{-n+q+1}\Bigr)^{\wedge}(\theta) .
\end{align}
For the definition of fractional derivatives, including the sign convention, we follow  \cite[Section 2.6]{K-book}. 

In particular, if $k$ is a non-negative even integer, $k\ne n-1$, then
  \begin{equation}\label{RY}
	A_{K,\theta}^{(k)}(0) = \frac{(-1)^{k/2} }{2\pi(n-1-k)}\Bigl(\|x\|_K^{-n+k+1}+\|-x\|_K^{-n+k+1}\Bigr)^{\wedge}(\theta) .
\end{equation}  
    If $K$ is  origin-symmetric, this yields the formula obtained in \cite{GKS}:
  \begin{equation}\label{GKS}
	A_{K,\theta}^{(k)}(0) = \frac{(-1)^{k/2} }{\pi(n-1-k)}\Bigl(\|x\|_K^{-n+k+1}\Bigr)^{\wedge}(\theta) .
    \end{equation}  
  
Let $K$ and $L$ be origin-symmetric star bodies in $\mathbb R^n$.   Following Lutwak~\cite{L} we say that $K$ is the {\it intersection body of $L$} and we write $K=IL$, if 
$$\rho_{K}(\theta) = \mathrm{vol}_{n-1}(L\cap \theta^\perp),\quad \mbox{for all } \theta \in S^{n-1}.$$
Since $\mathrm{vol}_{n-1}(L\cap \theta^\perp)= A_{L,\theta}(0)$, formula (\ref{GKS}) implies
$$\|\theta\|_K^{-1} = \frac{1}{\pi(n-1)} \left(\|\cdot\|_L^{-n+1}\right)^\wedge(\theta), \quad \mbox{for all }  \theta \in S^{n-1}.$$
 
 \section{Auxiliary results}

 \begin{lemma}\label{mainLemma}
 \label{calculation} Let $K$ be a convex body in $\mathbb R^n$, $n\ge 2$, that contains the origin in its interior, and let   $a= \min_{u\in S^{n-1}} \rho_{K} (u)$. Then for 
 any  $H \in \mathcal{H}_m^n$, the function
 \[
 M_H(t)=\int_{S^{n-1}} A_{K,\theta} (t) H(\theta)\,  d \theta 
 \]
 is analytic in $t \in (-a,a)$. Moreover, if $n$ is odd, then it is a polynomial in $t \in (-a,a)$. 
 \end{lemma}

 \begin{proof}
 	Recall that if \(f\) is an  integrable compactly supported function on
 	\(\mathbb R\), which is continuous at $t$, then
 	\begin{equation}\label{approx-identity}
 	f(t)=\lim_{\epsilon\to0^+}\frac{\epsilon}{2}
 	\int_{\mathbb R}|s-t|^{-1+\epsilon}f(s)\,ds.
 	\end{equation}
 	Applying this to \(M_H\), and using Fubini's theorem, we obtain
 	\begin{align*}
 		M_H(t)
 		&=
 		\lim_{\epsilon\to0^+}\frac{\epsilon}{2}
 		\int_{\mathbb R}|s-t|^{-1+\epsilon}M_H(s)\,ds\\
 		&=
 		\lim_{\epsilon\to0^+}\frac{\epsilon}{2}
 		\int_{S^{n-1}}\int_K
 		|\langle x,\theta\rangle-t|^{-1+\epsilon}
 		\,dx\,H(\theta)\,d\theta\\
 		&=
 		\lim_{\epsilon\to0^+}\frac{\epsilon}{2}
 		\int_{S^{n-1}}\int_0^{\rho_K(\xi)}r^{n-1}
 		\int_{S^{n-1}}
 		|r\langle\xi,\theta\rangle-t|^{-1+\epsilon}
 		H(\theta)\,d\theta\,dr\,d\xi .
 	\end{align*}
 	
For every \(\epsilon>0\) and every \(r>0\) with \(r\ne|t|\),
the Funk--Hecke formula gives
 	\[
 	\int_{S^{n-1}}
 	|r\langle\xi,\theta\rangle-t|^{-1+\epsilon}
 	H(\theta)\,d\theta
 	=
 	c_n H(\xi)
 	\int_{-1}^1
 	|rz-t|^{-1+\epsilon}
 	P_m^n(z)(1-z^2)^{\frac{n-3}{2}}\,dz,
 	\]
 	where \(c_n= | S^{n-2}|\).
 	Therefore
 	\begin{align*}
 		M_H(t)
 		&=
 		c_n\lim_{\epsilon\to0^+}\frac{\epsilon}{2}
 		\int_{S^{n-1}}H(\xi)
 		\int_0^{\rho_K(\xi)}r^{n-1}
 		\int_{-1}^1
 		|rz-t|^{-1+\epsilon}
 		P_m^n(z)(1-z^2)^{\frac{n-3}{2}}
 		\,dz\,dr\,d\xi .
 	\end{align*}
 	Making the change of variables $s=rz$ and interchanging the order of integration,
 	we obtain 
 	\begin{align*}&\int_0^{\rho_K(\xi)}r^{n-1}
 	\int_{-1}^1
 	|rz-t|^{-1+\epsilon}
 	P_m^n(z)(1-z^2)^{\frac{n-3}{2}}
 	\,dz\,dr \\
 	&=\int_0^{\rho_K(\xi)}r^{n-2}\int_{-r}^r
 	|s-t|^{-1+\epsilon}
 	P_m^n\left(\frac{s}{r}\right)
 	\left(1-\frac{s^2}{r^2}\right)^{\frac{n-3}{2}}
 	\,ds\,dr \\ &= 
 	\int_{-\rho_K(\xi)}^{\rho_K(\xi)}|s-t|^{-1+\epsilon}  
 	\int_{|s|}^{\rho_K(\xi)}
 	r^{n-2}
 	P_m^n\left(\frac{s}{r}\right)
 	\left(1-\frac{s^2}{r^2}\right)^{\frac{n-3}{2}}
 	\,dr\,ds.
 	\end{align*}

 	For \(\rho>0\), define
 	\[
 	\Phi_\rho(s)
 	=
 	\begin{cases}
 		\displaystyle
 		\int_{|s|}^{\rho}
 		r^{n-2}
 		P_m^n\!\left(\frac{s}{r}\right)
 		\left(1-\frac{s^2}{r^2}\right)^{\frac{n-3}{2}}\,dr,
 		& |s|<\rho,\\[3mm]
 		0, & |s|\ge\rho.
 	\end{cases}
 	\] 	
Then
 	\begin{align*}
 	M_H(t)
 	&=
 	c_n\lim_{\epsilon\to0^+}\frac{\epsilon}{2}
 \int_{S^{n-1}}H(\xi)	\int_{\mathbb R}|s-t|^{-1+\epsilon}\Phi_{\rho_K(\xi)}(s)\,ds\, d\xi\\
 &=
 c_n\lim_{\epsilon\to0^+}\frac{\epsilon}{2}
	\int_{\mathbb R}|s-t|^{-1+\epsilon}  \int_{S^{n-1}}H(\xi) \Phi_{\rho_K(\xi)}(s) \, d\xi\,ds.
 	\end{align*}
Let
 	\[
 	\Phi_H(s)
 	=
 	\int_{S^{n-1}}
 	H(\xi)\Phi_{\rho_K(\xi)}(s)\,d\xi.
 	\]
 Observe that the function \(\Phi_H\) is compactly supported and continuous on $(-a,a)$.   Applying   formula \eqref{approx-identity}, we obtain 
 	\[
 	M_H(t)
 	=
 	c_n\int_{S^{n-1}}H(\xi)\Phi_{\rho_K(\xi)}(t)\,d\xi,
 	\qquad |t|<a.
 	\]
 	
 	Since \(\rho_K(\xi)\ge a\), we can write, for $|t|<a$, 
 	\begin{align*}
 		\Phi_{\rho_K(\xi)}(t)
 		&=
 		\Phi_a(t)
 		+
 		\int_a^{\rho_K(\xi)}
 		r^{n-2}
 		P_m^n\!\left(\frac{t}{r}\right)
 		\left(1-\frac{t^2}{r^2}\right)^{\frac{n-3}{2}}\,dr.
 	\end{align*}
 	Therefore, for every \(|t|<a\),
 	\begin{align*}
 		M_H(t)
 		&=
 		c_n\Phi_a(t)\int_{S^{n-1}}H(\xi)\,d\xi\\
 		&\quad+
 		c_n\int_{S^{n-1}}H(\xi)
 		\int_a^{\rho_K(\xi)}
 		r^{n-2}
 		P_m^n\!\left(\frac{t}{r}\right)
 		\left(1-\frac{t^2}{r^2}\right)^{\frac{n-3}{2}}
 		\,dr\,d\xi.
 	\end{align*}
 	The second term is real analytic on \((-a,a)\). Moreover, it is a polynomial if \(n\) is
 	odd. 
 	
 	It remains to consider the first term. If \(m\ge1\), then
 	\[
 	\int_{S^{n-1}}H(\xi)\,d\xi=0,
 	\]
 	so this term vanishes. If \(m=0\), then \(P_0^n=1\), and
 	\[
 	\Phi_a(t)
 	=
 	\int_{|t|}^{a}
 	r^{n-2}
 	\left(1-\frac{t^2}{r^2}\right)^{\frac{n-3}{2}}\,dr
 	=
 	\frac{1}{n-1}(a^2-t^2)^{\frac{n-1}{2}}.
 	\]
 	This is real analytic on \((-a,a)\), and it is a polynomial when
 	\(n\) is odd.

 \end{proof}
  
\begin{lemma}\label{lem:continuation-separable}
Let \(K\subset\mathbb{R}^n\) be a convex body containing the origin in its
interior, and
suppose that
\[
-a<t_1<t_2<a,
\]
where
\(
a=\min_{\theta\in S^{n-1}}\rho_K(\theta).
\)
If
\begin{equation}\label{eq:separation-on-interval}
A_{K,\theta}(t)
=
\sum_{m=1}^N a_m(\theta)b_m(t)
\end{equation}
for every \(\theta\in S^{n-1}\) and every \(t\in(t_1,t_2)\), then there exist an integer $ N'\le N$,  linearly independent continuous functions $\tilde a_1, \ldots,  \tilde a_{ N'}$ on the sphere, and analytic functions $\tilde b_1,\ldots, \tilde b_{N'}$ on the interval $(-a,a)$ such that
\[
A_{K,\theta}(t)
=
\sum_{m=1}^{N'} \tilde a_m(\theta)\tilde b_m(t)
\]
holds for every \(\theta\in S^{n-1}\) and every \(t\in(-a,a)\).

\end{lemma}

\begin{proof}
For \(t\in(t_1,t_2)\), set

$$
F_t(\theta)=A_{K,\theta}(t).
$$

By the assumed separable representation, the space
$$
V=\operatorname{span}\{F_t:t\in(t_1,t_2)\}
$$
has dimension at most \(N\). Let \(N'=\dim V\), and choose
\(s_1,\dots,s_{N'}\in(t_1,t_2)\) such that
\(F_{s_1},\dots,F_{s_{N'}}\) form a basis of \(V\). Setting
$$
\tilde a_j(\theta)=F_{s_j}(\theta)=A_{K,\theta}(s_j),
$$
we obtain continuous, linearly independent functions
\(\tilde a_1,\dots,\tilde a_{N'}\), and for every \(t\in(t_1,t_2)\) there are unique
coefficients \(\tilde b_1(t),\dots,\tilde b_{N'}(t)\) such that

$$
A_{K,\theta}(t)=\sum_{j=1}^{N'} \tilde a_j(\theta)\tilde b_j(t).
$$
For simplicity of notation, from now on we will write $a_j$,  $b_j$, and \(N\)  instead of $\tilde a_j$, $\tilde b_j$, and \(N'\).

For a spherical harmonic \(H\), let
\[
M_H(t)
=
\int_{S^{n-1}}A_{K,\theta}(t)H(\theta)\,d\theta.
\]
By Lemma~\ref{mainLemma}, \(M_H\) is real analytic on \((-a,a)\).

We claim that there exist spherical harmonics \(H_1,\dots,H_N\) such that
the matrix
\[
\mathcal A
=
\left(
\int_{S^{n-1}}a_m(\theta)H_j(\theta)\,d\theta
\right)_{j,m=1}^N
\]
is invertible. Indeed, otherwise the vectors
\[
\left(
\int_{S^{n-1}}a_1(\theta)H(\theta)\,d\theta,
\dots,
\int_{S^{n-1}}a_N(\theta)H(\theta)\,d\theta
\right),
\]
as \(H\) ranges over all spherical harmonics, would span a proper subspace
of \(\mathbb{R}^N\). Hence there would exist scalars
\(\lambda_1,\dots,\lambda_N\), not all zero, such that
\[
\int_{S^{n-1}}
\left(
\sum_{m=1}^N\lambda_m a_m(\theta)
\right)
H(\theta)\,d\theta
=
0
\]
for every spherical harmonic \(H\). By completeness of the spherical
harmonics in \(L^2(S^{n-1})\), this would imply
\[
\sum_{m=1}^N\lambda_m a_m=0,
\]
contradicting the linear independence of \(a_1,\dots,a_N\).

Define functions \(\beta_1,\dots,\beta_N\) on \((-a,a)\) by
\[
\begin{pmatrix}
\beta_1(t)\\
\vdots\\
\beta_N(t)
\end{pmatrix}
=
\mathcal A^{-1}
\begin{pmatrix}
M_{H_1}(t)\\
\vdots\\
M_{H_N}(t)
\end{pmatrix}.
\]
Each \(\beta_m\) is real analytic on \((-a,a)\).

For \(t\in(t_1,t_2)\), equation
\eqref{eq:separation-on-interval} gives
\[
\begin{pmatrix}
M_{H_1}(t)\\
\vdots\\
M_{H_N}(t)
\end{pmatrix}
=
\mathcal A
\begin{pmatrix}
b_1(t)\\
\vdots\\
b_N(t)
\end{pmatrix}.
\]
Therefore,
\[
\beta_m(t)=b_m(t),
\qquad
t\in(t_1,t_2),\quad m=1,\dots,N.
\]

Let \(H\) be an arbitrary spherical harmonic and set
\[
R_H(t)
=
M_H(t)
-
\sum_{m=1}^N
\beta_m(t)
\int_{S^{n-1}}a_m(\theta)H(\theta)\,d\theta.
\]
The function \(R_H\) is real analytic on \((-a,a)\) and vanishes on
\((t_1,t_2)\). Therefore,
\[
R_H(t)=0,
\qquad t\in(-a,a).
\]

Fix \(t\in(-a,a)\). It follows that the function
\[
\theta\longmapsto
A_{K,\theta}(t)
-
\sum_{m=1}^N a_m(\theta)\beta_m(t)
\]
is orthogonal to every spherical harmonic. By completeness of the
spherical harmonics, it vanishes almost everywhere. Since it is continuous
in \(\theta\), it vanishes everywhere. Hence
\[
A_{K,\theta}(t)
=
\sum_{m=1}^N a_m(\theta)\beta_m(t),
\qquad
\theta\in S^{n-1},\quad t\in(-a,a).
\]
Renaming \(\beta_m\) as \(b_m\) completes the proof.
\end{proof}

The next corollary is a particular case of  the previous lemma  when the functions $b_m$ are polynomials.

\begin{corollary}\label{cor:continuation-polynomial}
Let \(K\subset\mathbb{R}^n\),  \(a\), $t_1$, and $t_2$ be as in
Lemma~\ref{lem:continuation-separable}. Suppose that
\[
A_{K,\theta}(t)
=
\sum_{m=0}^N a_m(\theta)t^m
\]
for every \(\theta\in S^{n-1}\) and every \(t\in(t_1,t_2)\). Then
\[
A_{K,\theta}(t)
=
\sum_{m=0}^N a_m(\theta)t^m
\]
for every \(\theta\in S^{n-1}\) and every \(t\in(-a,a)\).

\end{corollary}

We will also need the following extension lemma.

\begin{lemma}
	\label{No Smooth Extension}
	Let \(p>0\) be non-integer, {$\varepsilon> 0$}, and set \(k=\lfloor p\rfloor\). Suppose that
	\[
	f:(\tau-\varepsilon,\tau]\to\mathbb{R}
	\]
	satisfies
	\[
	f(t)=C(\tau-t)^p(1+o(1))
	\qquad\text{as }t\to\tau^-,
	\]
	where \(C>0\). Then \(f\) does not admit a \(C^{k+1}\)-extension to any
	two-sided neighborhood of \(\tau\).
\end{lemma}

\begin{proof}
	Suppose, to the contrary, that there exist \(\delta>0\) and a function
	\[
	g\in C^{k+1}\bigl((\tau-\delta,\tau+\delta)\bigr)
	\]
	such that
	\[
	g(t)=f(t)
	\qquad
	\text{for }
	t\in(\tau-\min\{\delta,\varepsilon\},\tau].
	\]
	
	Since \(k=\lfloor p\rfloor\) and \(p\notin\mathbb{Z}\), we have
	\[
	k<p<k+1.
	\]
	
	We claim that
	\[
	g^{(j)}(\tau)=0,
	\qquad j=0,\dots,k.
	\]
	Suppose otherwise, and let \(j\leq k\) be the smallest index such that
	\(g^{(j)}(\tau)\neq 0\). Taylor's theorem  gives
	\[
	g(t)
	=
	\frac{g^{(j)}(\tau)}{j!}(t-\tau)^j
	+
	o\bigl(|t-\tau|^j\bigr)
	\qquad\text{as }t\to\tau^-.
	\]
	Therefore,
	\[
	\lim_{t\to\tau^-}
	\frac{g(t)}{(\tau-t)^j}
	=
	(-1)^j\frac{g^{(j)}(\tau)}{j!}
	\neq 0.
	\]
	On the other hand, since \(j\leq k<p\),
	\[
	\lim_{t\to\tau^-}
	\frac{g(t)}{(\tau-t)^j}
	=
	\lim_{t\to\tau^-}
	\frac{f(t)}{(\tau-t)^j}
	=
	\lim_{t\to\tau^-}
	C(\tau-t)^{p-j}(1+o(1))
	=
	0,
	\]
	which is a contradiction. Hence
	\[
	g^{(j)}(\tau)=0,
	\qquad j=0,\dots,k.
	\]
	
	Since \(g\in C^{k+1}\), Taylor's theorem yields
	\[
	|g(t)|\leq M|t-\tau|^{k+1}
	\]
	for all \(t<\tau\) sufficiently close to \(\tau\), where \(M>0\) is a
	constant. Consequently,
	\[
	\left|
	\frac{f(t)}{(\tau-t)^{k+1}}
	\right|
	=
	\left|
	\frac{g(t)}{(\tau-t)^{k+1}}
	\right|
	\leq M.
	\]
	However,
	\[
	\frac{f(t)}{(\tau-t)^{k+1}}
	=
	C(\tau-t)^{p-k-1}(1+o(1)),
	\]
	which tends to infinity as \(t\to\tau^-\), since \(p-k-1<0\).
	This contradicts the preceding bound.
	
	Therefore, no such \(C^{k+1}\)-extension exists.
\end{proof}

Finally, we prove the following smoothness lemma.

 \begin{lemma}\label{lem:smoothness-from-sections}
 	Let \(K\subset\mathbb R^n\), \(n\ge2\), be a convex body containing
 	the origin in its interior. Suppose that there exists
 	\(\varepsilon>0\) such that \(A_{K,\theta}(t)\) is \(C^\infty\) in
 	\(t\) for \(|t|<\varepsilon\), and, for every integer \(j\ge0\),
 	\[
 	(\theta,t)\longmapsto A_{K,\theta}^{(j)}(t)
 	\]
 	is continuous on
 	\(S^{n-1}\times(-\varepsilon,\varepsilon)\).
 	Then \(K\) is infinitely smooth.
 \end{lemma} 
 
 \begin{proof}
 	To prove that $K$ is infinitely smooth, we will show that $\|\cdot\|_K\in C^{m} (S^{n-1})$ for every $m\in \mathbb N$. 
To this end, we will represent a positive power of
\(\|\cdot\|_K\) as a spherical convolution of a continuous function
with an \(m\)-times differentiable zonal kernel; see
\eqref{eq:half-integer} below.

 We first observe that, for every \(\theta\in S^{n-1}\), the
 fractional derivative
 \[
 q\longmapsto A_{K,\theta}^{(q)}(0)
 \]
 extends to an entire function of \(q\); see
 \cite[Section~2.6]{K-book}. Moreover, \(A_{K,\theta}^{(q)}(0)\) is bounded uniformly in
 \(\theta\) when \(q\) ranges over a bounded subset of \(\mathbb C\). Indeed, the latter follows
 from the joint continuity of
 \(A_{K,\theta}^{(j)}(t)\), \(j\ge0\), and the compactness of
 \(S^{n-1}\).

 	For $-1<\Re q<0$, formula~\eqref{FracDer-RY}, which in this range
 	does not require any smoothness assumption on $K$, can be written as
 	\[
 	A_{K,\theta}^{(q)}(0)
 	=
 	\frac{1}{2\pi(n-1-q)}
 	\left(
 	e^{-i\pi q/2}\left( \|\cdot\|_K^{-n+1+q}\right)^\wedge(\theta)
 	+
 	e^{i\pi q/2}\left( \|\cdot\|_K^{-n+1+q}\right)^\wedge(-\theta)
 	\right).
 	\]
 	Replacing $\theta$ by $-\theta$ and solving the resulting system of
 	two equations, we obtain
 	\begin{equation}\label{eq:single-Fourier-sine}
 		\sin(\pi q)\left( \|\cdot\|_K^{-n+1+q}\right)^\wedge(\theta)
 		=
 		i\pi(n-1-q)
 		\left(
 		e^{-i\pi q/2}A_{K,\theta}^{(q)}(0)
 		-
 		e^{i\pi q/2}A_{K,-\theta}^{(q)}(0)
 		\right).
 	\end{equation}
 	
 	For $q\in\mathbb C$, let $F_q$ be the homogeneous function on
 	$\mathbb R^n\setminus\{0\}$ of degree $-q-1$ whose restriction to
 	$S^{n-1}$ is
 	\begin{equation}\label{eq:def-F}
 		F_q(\theta)
 		=
 		i\pi(n-1-q)
 		\left(
 		e^{-i\pi q/2}A_{K,\theta}^{(q)}(0)
 		-
 		e^{i\pi q/2}A_{K,-\theta}^{(q)}(0)
 		\right), \qquad \theta\in S^{n-1}.
 	\end{equation}
 	Thus for  	$-1<\Re q<0$, \eqref{eq:single-Fourier-sine} gives  
 	\begin{equation}\label{eq:Fq-equality}
 	\sin(\pi q)\langle \left( \|\cdot\|_K^{-n+1+q}\right)^\wedge, \varphi\rangle=\langle F_q, \varphi\rangle,
 	\end{equation}
 	for any test function  	$\varphi\in\mathcal S(\mathbb R^n)$.
 	
 	We claim that this identity extends to the strip
 	\begin{equation}\label{eq:q-strip}
 		-1<\Re q<n-1.
 	\end{equation}
  Indeed, in this strip $F_q$ is locally
 	integrable, since it is homogeneous of degree $-q-1$. Hence, for any
 	$\varphi\in\mathcal S(\mathbb R^n)$,
 	\[
 	\langle F_q,\varphi\rangle
 	=	\int_{\mathbb R^n}
 	F_q(x) \varphi(x)\,dx=
 	\int_{S^{n-1}}F_q(u)
 	\int_0^\infty r^{n-q-2}\varphi(ru)\,dr\,du,
 	\]
 	and  differentiation under the integral sign therefore shows that this is analytic in $q$ on the strip \eqref{eq:q-strip}.
 	On the other hand,
 	\[
 	\langle  \left( \|\cdot\|_K^{-n+1+q}\right)^\wedge,\varphi\rangle
 	=
 	\langle    \|\cdot\|_K^{-n+1+q} ,\widehat\varphi\rangle
 	=
 	\int_{\mathbb R^n}
 	\|x\|_K^{-n+1+q}\widehat\varphi(x)\,dx
 	\]
 	is analytic for $\Re q>-1$. Therefore, for every test function
 	$\varphi$, both sides of \eqref{eq:Fq-equality}
 	are analytic on the strip \eqref{eq:q-strip}. Since they agree for
 	$-1<\Re q<0$, they agree for  $-1<\Re q<n-1$.
 	
 Using  $\widehat \varphi$ instead of $\varphi$ in \eqref{eq:Fq-equality}, we obtain
 $$\sin(\pi q)\langle \left( \|\cdot\|_K^{-n+1+q}\right)^\wedge, \widehat\varphi\rangle=\langle F_q, \widehat\varphi\rangle$$
 	
By the definition of the Fourier transform of distributions \eqref{action} and the formula 
 	\[
 	\widehat{\widehat f\,\,}(x)=(2\pi)^n f(-x),
 	\]
 	we obtain
 	\begin{equation}\label{eq:Fq-second}
 		(2\pi)^n\sin(\pi q)
 	\langle	\|-x\|_K^{-n+1+q},\varphi(x)\rangle =\langle \widehat F_q, \varphi \rangle,
 		\qquad
 		-1<\Re q<n-1.
 	\end{equation}
 	
 	For \(n-2<\Re q<n-1\), formula 	\eqref{eq:GYY}
 	shows that \(\widehat F_q\) is represented  by a 
 	continuous function on $\mathbb R^{n}\setminus \{0\}$ that is homogeneous of degree $-n+q+1$ . Hence \eqref{eq:Fq-second} becomes the
 	pointwise identity
 	\begin{align}\label{eq:Fq-GYY}
 		&(2\pi)^n\sin(\pi q)
 		\|-\theta\|_K^{-n+1+q}
 		\notag\\
 		&\quad=
 		\Gamma(n-1-q)
 		\cos\frac{\pi(n-1-q)}2
 		\int_{S^{n-1}}
 		|\langle\theta,u\rangle|^{-n+q+1}F_q(u)\,du
 		\notag\\
 		&\qquad
 		-i\Gamma(n-1-q)
 		\sin\frac{\pi(n-1-q)}2
 		\int_{S^{n-1}}
 		\operatorname{sgn}\langle\theta,u\rangle
 		|\langle\theta,u\rangle|^{-n+q+1}F_q(u)\,du.
 	\end{align}
 	For every fixed \(\theta\in S^{n-1}\), the two integrals in
 	\eqref{eq:Fq-GYY} are analytic in \(q\) in the half-plane
 	\(
 	\Re q>n-2.
 	\)
 	Indeed, \(\Re(q-n+1)>-1\), while \(F_q(u)\) is entire in \(q\)
 	and locally bounded uniformly in \(u\). It follows that the
 	right-hand side of \eqref{eq:Fq-GYY} is analytic in
 	\[
 	\Omega=
 	\{q\in\mathbb C:\Re q>n-2\}
 	\setminus\{n-1,n,n+1,\ldots\}.
 	\]
 	
 The left-hand side of \eqref{eq:Fq-GYY} is an entire function of the variable \(q\). Since \(\Omega\) is connected
 and the two sides agree for \(n-2<\Re q<n-1\), the identity theorem
 implies that \eqref{eq:Fq-GYY} holds for every \(q\in\Omega\).

 	Fix an integer $m\geq0$ and set
 	\[
 	r=m+\frac12,
 	\qquad
 	q=n-1+r.
 	\]
 	Since $q$ is a half-integer, none of the gamma factors in
 	\eqref{eq:Fq-GYY} has a pole, and $\sin(\pi q)\neq0$.
 	Substituting $q=n-1+r$ into \eqref{eq:Fq-GYY} gives
 	\begin{align}
 		\|-\theta\|_K^r
 		&=
 	\frac{	\Gamma(-r)}{(2\pi)^n\sin(\pi q)}\cos\frac{\pi r}{2}
 		\int_{S^{n-1}}
 		|\langle\theta,u\rangle|^rF_q(u)\,du
 		\notag\\
 		&\qquad
 		+\frac{	i\Gamma(-r)}{(2\pi)^n\sin(\pi q)}\sin\frac{\pi r}{2}
 		\int_{S^{n-1}}
 		\operatorname{sgn}\langle\theta,u\rangle
 		|\langle\theta,u\rangle|^rF_q(u)\,du.
 		\label{eq:half-integer}
 	\end{align}
 	
 	Since $r=m+\frac12$, both functions
 	\[
 	s\longmapsto |s|^r,
 	\qquad
 	s\longmapsto \operatorname{sgn}(s)|s|^r
 	\]
 	belong to $C^m(\mathbb R)$. Since $F_q$ is continuous on the sphere $S^{n-1}$,
 	differentiation under the integrals in \eqref{eq:half-integer} gives
 	\[
 	\|-\theta\|_K^r\in C^m(S^{n-1}).
 	\]
 	Hence 
 	\[
 	\|\cdot\|_K\in C^m(S^{n-1}).
 	\]
 	As $m$ is arbitrary,
 	\[
 	\|\cdot\|_K\in C^\infty(S^{n-1}),
 	\]
 	and therefore $K$ is infinitely smooth.
 \end{proof}

 \section{Applications to polynomially and separably integrable bodies}\label{Section-polynomial-separable}

Before we state our first theorem, let us recall the following fact proved in \cite[Theorem~3.7]{kmy}.

\begin{theorem}
\label{thm A Poly Ellipsoid}
    Let $n$ be an odd positive integer. Let $K$ be an infinitely smooth convex body
in $\mathbb{R}^n$, containing the origin in its interior. Suppose that there exists a number $N \geq n$ such that for every $\theta \in S^{n-1}$ and every $m \geq N$ we have
\begin{align*}
    \left.\frac{d^m}{dt^m}\right|_{t=0}A_{K,\theta} (t) = 0.
\end{align*}
Then $K$ is an ellipsoid.
\end{theorem}

Theorem \ref{thm A Poly Ellipsoid}  was used in \cite{kmy} to show that the only polynomially integrable bodies in odd dimensions are ellipsoids.
The following theorem is an extension of this result to bodies with separable section functions near the origin. No smoothness assumption is needed, since Lemma \ref{lem:smoothness-from-sections} shows that local separability of the section function implies that \(K\) is infinitely smooth.
 \begin{theorem}
	Let \(n\) be odd, and let \(K\subset\mathbb{R}^n\) be a 
convex body containing the origin in its interior. If \(K\) has a locally
separable section function near the origin, then \(K\) is an ellipsoid.
\end{theorem}

\begin{proof}
As before, set
	\(
	a=\min_{\theta\in S^{n-1}}\rho_K(\theta).
	\)	
	By Lemma~\ref{lem:continuation-separable}, we may assume that
	\begin{equation}\label{eq:global-separation-odd}
		A_{K,\theta}(t)
		=
		\sum_{m=1}^N a_m(\theta)b_m(t),
		\qquad
		\theta\in S^{n-1},\quad -a<t<a,
		\end{equation}
where	\(a_1,\dots,a_N\) are linearly independent and continuous on $S^{n-1}$, and the   functions $b_m$ are real analytic on \((-a,a)\). Moreover, Lemma \ref{lem:smoothness-from-sections} implies that $K$ is infinitely smooth. 
	
	As in the proof of Lemma~\ref{lem:continuation-separable}, there exist
	spherical harmonics \(H_1,\dots,H_N\) such that the matrix
	\[
	\mathcal A
	=
	\left(
	\int_{S^{n-1}}a_m(\theta)H_j(\theta)\,d\theta
	\right)_{j,m=1}^N
	\]
	is invertible. For \(j=1,\dots,N\), define
	\[
	M_j(t)
	=
	\int_{S^{n-1}}A_{K,\theta}(t)H_j(\theta)\,d\theta.
	\]
	Using \eqref{eq:global-separation-odd}, we obtain
	\[
	M_j(t)
	=
	\sum_{m=1}^N
	b_m(t)
	\int_{S^{n-1}}a_m(\theta)H_j(\theta)\,d\theta.
	\]
	Equivalently,
	\[
	\begin{pmatrix}
		M_1(t)\\
		\vdots\\
		M_N(t)
	\end{pmatrix}
	=
	\mathcal A
	\begin{pmatrix}
		b_1(t)\\
		\vdots\\
		b_N(t)
	\end{pmatrix},
	\]
	and hence
	\begin{equation}\label{eq:b-from-moments}
		\begin{pmatrix}
			b_1(t)\\
			\vdots\\
			b_N(t)
		\end{pmatrix}
		=
		\mathcal A^{-1}
		\begin{pmatrix}
			M_1(t)\\
			\vdots\\
			M_N(t)
		\end{pmatrix}.
	\end{equation}
	
	Since \(n\) is odd, Lemma~\ref{mainLemma} implies that each \(M_j(t)\) is
	a polynomial in \(t\). It follows from \eqref{eq:b-from-moments} that each
	\(b_m(t)\) is also a polynomial. Since only finitely many functions are
	involved, there exists an integer \(L\) such that
	\[
	\deg b_m\le L,
	\qquad m=1,\dots,N.
	\]
	Therefore, by \eqref{eq:global-separation-odd}, for every
	\(\theta\in S^{n-1}\), the function \(A_{K,\theta}(t)\) is a polynomial
	in \(t\) of degree at most \(L\) on \((-a,a)\). Consequently,
	\[
	A_{K,\theta}^{(k)}(0)=0
	\qquad
	\text{for every }\theta\in S^{n-1}
	\text{ and every }k>L.
	\]
	
	By Theorem~\ref{thm A Poly Ellipsoid},  \(K\) is an ellipsoid.
\end{proof}

We will now turn to the case of even dimensions. First we treat the case of local separability near the boundary.

 \begin{theorem}
  	Let \(K\) be an infinitely smooth convex body in \(\mathbb{R}^n\), where
  	\(n\) is even. Suppose that there exists \(\varepsilon>0\) such that
  	\begin{equation}\label{Separable}
  		A_{K,\theta}(t)
  		=
  		\sum_{m=1}^N a_m(\theta)b_m(t)
  	\end{equation}
  	for every \(\theta\in S^{n-1}\) and every
  	\[
  	t\in\bigl(h_K(\theta)-\varepsilon,h_K(\theta)\bigr).
  	\]
  	Then \(K\) is a Euclidean ball centered at the origin.
  \end{theorem}
  
  \begin{proof}
  	Since \(K\) is infinitely smooth, it follows from
  	\cite[Lemma~2.1]{akry} that there exists a dense subset
  	\(\Sigma\subset S^{n-1}\) such that, for every \(\theta\in\Sigma\),
  	\begin{equation}\label{eq:boundary-asymptotic}
  		A_{K,\theta}(t)
  		=
  		c(\theta)\bigl(h_K(\theta)-t\bigr)^{\frac{n-1}{2}}
  		\bigl(1+o(1)\bigr),
  		\qquad
  		t\to h_K(\theta)^-,
  	\end{equation}
  	where \(c(\theta)>0\).
  	
  	Suppose, to the contrary, that \(h_K\) is not constant on \(S^{n-1}\).
  	Then there exists
  	\(\theta_0\in S^{n-1}\) such that \(h_K\) is not constant on any
  	neighborhood of \(\theta_0\). 
  	
  	Set
  	\[
  	w_K(\theta)=h_K(\theta)+h_K(-\theta),
  	\]
  	the width of \(K\) in the direction \(\theta\). Since \(K\) has non-empty
  	interior, \(w_K(\theta_0)>0\). By continuity, we may choose a connected
  	open neighborhood \(U\) of \(\theta_0\) such that \(h_K\) is not
  	constant on \(U\) and
  	\begin{equation}\label{eq:small-oscillation}
  		\sup_{\theta\in U}h_K(\theta)
  		-
  		\inf_{\theta\in U}h_K(\theta)
  		<
  		\min\left\{\varepsilon,\inf_{\theta\in U}w_K(\theta)\right\}.
  	\end{equation}
  	Since \(U\) is connected and \(h_K\) is continuous and nonconstant on
  	\(U\), the set \(h_K(U)\) is a nondegenerate interval. Moreover,
  	\(\Sigma\cap U\) is dense in \(U\), and hence \(h_K(\Sigma\cap U)\) is
  	dense in \(h_K(U)\). We may therefore choose
  	\[
  	\theta_1,\dots,\theta_{N+1}\in\Sigma\cap U,
  	\]
  (recall that $N$ is the number of summands in \eqref{Separable}),
  	such that the numbers
  	\[
  	h_K(\theta_1),\dots,h_K(\theta_{N+1})
  	\]
  	are pairwise distinct. After relabelling, assume that
  	\[
  	h_K(\theta_1)
  	<
  	h_K(\theta_2)
  	<
  	\cdots
  	<
  	h_K(\theta_{N+1}).
  	\]
  	For brevity, write
  	\[
  	h_j=h_K(\theta_j),
  	\qquad j=1,\dots,N+1.
  	\]
  	By \eqref{eq:small-oscillation},
  	\begin{equation}\label{eq:support-spread}
  		h_{N+1}-h_1<\varepsilon.
  	\end{equation}
  	
  	The \(N+1\) vectors
  	\[
  	v_j
  	=
  	\begin{pmatrix}
  		a_1(\theta_j)\\
  		\vdots\\
  		a_N(\theta_j)
  	\end{pmatrix}
  	\in\mathbb{R}^N,
  	\qquad j=1,\dots,N+1,
  	\]
  	are linearly dependent. Hence there exist
  	\(\lambda_1,\dots,\lambda_{N+1}\in\mathbb{R}\), not all zero, such that
  	\[
  	\sum_{j=1}^{N+1}\lambda_jv_j=0.
  	\]
  	Equivalently,
  	\begin{equation}\label{eq:coefficient-relation}
  		\sum_{j=1}^{N+1}\lambda_j a_m(\theta_j)=0,
  		\qquad m=1,\dots,N.
  	\end{equation}
  	
  	Let
  	\(
  	j_*=\min\{j:\lambda_j\neq0\}.
  	\)
  	We claim that
  	\begin{equation}\label{eq:section-relation}
  		\sum_{j=j_*}^{N+1}\lambda_jA_{K,\theta_j}(t)=0
  	\end{equation}
  	for every
  	\(
  	t\in\bigl(h_{N+1}-\varepsilon,h_{j_*}\bigr).
  	\)
  	Indeed, if \(j\ge j_*\) and
  	\[
  	h_{N+1}-\varepsilon<t<h_{j_*},
  	\]
  	then
  	\[
  	t>h_{N+1}-\varepsilon\ge h_j-\varepsilon
  	\]
  	and
  	\[
  	t<h_{j_*}\le h_j.
  	\]
  	Thus
  	\[
  	t\in(h_j-\varepsilon,h_j),
  	\]
  	so \eqref{Separable} is valid for every \(j\ge j_*\). It follows from
  	\eqref{eq:coefficient-relation} that
  	\[
  	\begin{aligned}
  		\sum_{j=j_*}^{N+1}\lambda_jA_{K,\theta_j}(t)
  		&=
  		\sum_{j=j_*}^{N+1}\lambda_j
  		\sum_{m=1}^N a_m(\theta_j)b_m(t)\\
  		&=
  		\sum_{m=1}^N b_m(t)
  		\sum_{j=j_*}^{N+1}\lambda_j a_m(\theta_j)
  		=0.
  	\end{aligned}
  	\]
  	The interval in \eqref{eq:section-relation} is non-empty because
  	\[
  	h_{N+1}-h_{j_*}
  	\le h_{N+1}-h_1
  	<\varepsilon.
  	\]
  	
  	Since \(\lambda_{j_*}\neq0\), equation
  	\eqref{eq:section-relation} gives
  	\begin{equation}\label{eq:extension-relation}
  		A_{K,\theta_{j_*}}(t)
  		=
  		-\frac{1}{\lambda_{j_*}}
  		\sum_{j=j_*+1}^{N+1}
  		\lambda_jA_{K,\theta_j}(t),
  		\qquad
  		h_{N+1}-\varepsilon<t<h_{j_*}.
  	\end{equation}
  	
  	For every \(j>j_*\), we have
  	\[
  	0<h_j-h_{j_*}
  	\le h_{N+1}-h_1
  	<
  	\inf_{\theta\in U}w_K(\theta)
  	\le w_K(\theta_j).
  	\]
  	Consequently,
  	\[
  	-h_K(-\theta_j)
  	=
  	h_j-w_K(\theta_j)
  	<
  	h_{j_*}
  	<
  	h_j.
  	\]
  	Thus \(h_{j_*}\) lies in the interior of the support of
  	\(A_{K,\theta_j}\). Since \(K\) is infinitely smooth,
  	\(A_{K,\theta_j}\) is \(C^\infty\) in a neighbourhood of \(h_{j_*}\).
  	Therefore,
  	\[
  	g(t)
  	=
  	-\frac{1}{\lambda_{j_*}}
  	\sum_{j=j_*+1}^{N+1}
  	\lambda_jA_{K,\theta_j}(t)
  	\]
  	is \(C^\infty\) in a neighbourhood of \(h_{j_*}\).  
  	
  	By \eqref{eq:extension-relation}, \(g\) agrees with
  	\(A_{K,\theta_{j_*}}\) on a left neighbourhood of \(h_{j_*}\). On the
  	other hand, since \(\theta_{j_*}\in\Sigma\),
  	\eqref{eq:boundary-asymptotic} gives
  	\[
  	A_{K,\theta_{j_*}}(t)
  	=
  	c(\theta_{j_*})
  	\bigl(h_{j_*}-t\bigr)^{\frac{n-1}{2}}
  	\bigl(1+o(1)\bigr),
  	\qquad
  	t\to h_{j_*}^-.
  	\]
  	Because \(n\) is even, \((n-1)/2\) is not an integer. Lemma
  	\ref{No Smooth Extension} therefore implies that
  	\(A_{K,\theta_{j_*}}\) cannot agree on a left neighbourhood of
  	\(h_{j_*}\) with a \(C^\infty\) function defined across \(h_{j_*}\).
  	This contradicts the existence of \(g\).
  	
  	Hence \(h_K\) is constant on \(S^{n-1}\).
  	Therefore \(K\) is a Euclidean ball centered at the origin.
  \end{proof}

We will now prove an analogue of the previous theorem in the case of local separability near the origin. No smoothness of $K$ is required.

\begin{theorem}\label{thm:even-near-origin}
	Let $n$ be even, and let $K\subset\R^n$ be a
	convex body containing the origin in its interior. Suppose that $K$ has a locally separable  section function near the origin, i.e., there exists $\eps\in (0,\min_{\theta\in S^{n-1}}h_K(\theta))$ such that 
    \begin{equation*} 
		A_{K,\theta}(t)
		=
		\sum_{m=1}^N a_m(\theta)b_m(t),
		\qquad
		\theta\in S^{n-1},\quad -\eps<t<\eps.
		\end{equation*}
    
    Then $K$ is a Euclidean ball centered at the origin. 
	
\end{theorem}

\begin{proof} Recall that $a=\min_{\theta\in S^{n-1}}\rho_K(\theta)=\min_{\theta\in S^{n-1}}h_K(\theta).$
	By Lemma~\ref{lem:continuation-separable}, we may assume that
	\begin{equation}\label{eq:sep}
		A_{K,\theta}(t)
		=
		\sum_{m=1}^N a_m(\theta)b_m(t),
		\qquad
		\theta\in S^{n-1},\quad -a<t<a,
		\end{equation}
where	\(a_1,\dots,a_N\) are continuous and linearly independent, and the   functions $b_m$ are real analytic on \((-a,a)\). Moreover, by Lemma \ref{lem:smoothness-from-sections}, the body $K$ is infinitely smooth.

The proof is divided into five steps.

\medskip
\noindent
 
 \textbf{Step 1.} We  use Fourier analysis to show that there is a finite-dimensional subspace $W$ of  $ C_{\mathrm{even}}(S^{n-1})=\{f\in C(S^{n-1}):f(\theta)=f(-\theta)\}$, and functions $w_k\in W$ such that
 \begin{equation}\label{norm-vs-S_p}
 	\|\theta\|_K^{-n+k+1}+\|-\theta\|_K^{-n+k+1}
 	=\int_{S^{n-1}}|\langle\theta,u\rangle|^{-n+k+1}w_k(u)\,du,
 \end{equation}
 for all even $k\ge n$.

 Using formula \eqref{RY} with even $k\ge n$ and the  separability condition \eqref{eq:sep}, we obtain
 \begin{align}\label{eq:RY}
 	\Bigl(\|x\|_K^{-n+k+1}+\|-x\|_K^{-n+k+1}\Bigr)^{\wedge}(\theta)& =2\pi(n-1-k)(-1)^{k/2} A_{K,\theta}^{(k)}(0) \nonumber\\
 	&	=2\pi(n-1-k)(-1)^{k/2}
 	\sum_{m=1}^N a_m(\theta)b_m^{(k)}(0).
 \end{align}
 Since the left-hand side of \eqref{eq:RY} is an even function of $\theta$, in the right-hand side we can replace $a_m$ by $e_m$, where 
 $$ e_m(\theta)=\frac12\bigl(a_m(\theta)+a_m(-\theta)\bigr), \qquad m=1,\ldots,N.
 $$
 Extending both sides of \eqref{eq:RY} to $\mathbb R^n\setminus\{0\}$ as homogeneous functions of degree $-k-1$, we can interpret them as homogeneous distributions that agree on \(\mathbb R^n\setminus\{0\}\). Hence their difference \(D\) is
 supported at the origin.  
 Moreover, \(D\) is even and homogeneous of degree \(-k-1\).
 Since \(k-n+1\) is odd, every homogeneous distribution supported
 at the origin of degree
 \[
 -k-1=-n-(k-n+1)
 \]
 is a linear combination of derivatives of the delta distribution
 of total order \(k-n+1\), and is therefore odd. Thus \(D=0\).
 Consequently, \eqref{eq:RY} holds on
 \(\mathbb R^n\) in the sense of distributions. Applying the Fourier
 transform, we obtain
 
 \begin{multline*} 
 	(2\pi)^n	\Bigl(\|\theta\|_K^{-n+k+1}+\|-\theta\|_K^{-n+k+1}\Bigr) \\ 	=2\pi(n-1-k)(-1)^{k/2}
 	\sum_{m=1}^N \left(e_m(y/|y|)|y|^{-k-1}\right)^\wedge(\theta) b_m^{(k)}(0),
 \end{multline*}
 for $\theta\in S^{n-1}$.
 
 Formula \eqref{p-odd-negative} yields
 \begin{multline*} 
 	\|\theta\|_K^{-n+k+1}+\|-\theta\|_K^{-n+k+1}   \\	=	(2\pi)^{-n+1}(n-1-k) (-1)^{k/2} \frac{\pi (-1)^{(n-k-2)/2}}{2(-n+k+1)!}
 	\sum_{m=1}^N \int_{S^{n-1}} |\langle \theta, u\rangle|^{-n+k+1}  e_m(u) b_m^{(k)}(0)\, du.
 \end{multline*}

 Set
 \[
 W=\operatorname{span}\{e_1,\dots,e_N\}
 \subset C_{\mathrm{even}}(S^{n-1})
 \]
 and
 \[
 w_k(u)=(2\pi)^{-n+1}  \frac{\pi (-1)^{n/2}}{2(-n+k)!}\sum_{m=1}^N b_m^{(k)}(0)e_m(u).
 \]
 Then $w_k\in W$, and \eqref{norm-vs-S_p} follows.

\medskip
\noindent
\textbf{Step 2.} We now study the operators
$S_p:C_{\mathrm{even}}(S^{n-1})\to C_{\mathrm{even}}(S^{n-1})$, $p>0$, defined by 
$$ S_pf(\theta)=\frac{1}{C_p}\int_{S^{n-1}}|\langle\theta,u\rangle|^p f(u)\,du,
$$
where 
\[
 C_p=\int_{S^{n-1}}|\langle\theta,u\rangle|^p\,du.
\]
Our goal is to show that, for every even function \(f\in C(S^{n-1})\),
\[
S_p f\longrightarrow f
\qquad\text{uniformly on }S^{n-1}.
\]

By rotational invariance, $C_p$ is independent of $\theta$. Taking
$\theta=e_1$ and using spherical coordinates gives
\[
 \begin{aligned}
 C_p
 &=2|S^{n-2}|\int_0^1 t^p(1-t^2)^{(n-3)/2}\,dt\\
 &=|S^{n-2}|\int_0^1
 v^{(p-1)/2}(1-v)^{(n-3)/2}\,dv\\
 &=|S^{n-2}|B\left(\frac{p+1}{2},\frac{n-1}{2}\right)\\
  &=|S^{n-2}| \frac{\Gamma\left(\frac{p+1}{2} \right)\Gamma\left(\frac{n-1}{2} \right)}{\Gamma\left(\frac{p+n}{2} \right)} 
 \end{aligned}
\]
Using the asymptotic formula
\[
\frac{\Gamma(x+\alpha)}{\Gamma(x+\beta)}
\sim x^{\alpha-\beta}
\qquad\text{as }x\to\infty,
\]
for fixed real numbers \(\alpha,\beta\), which is a consequence of Stirling's formula, we obtain
\[
C_p
\sim
|S^{n-2}|
\Gamma\left(\frac{n-1}{2}\right)
\left(\frac{2}{p}\right)^{(n-1)/2}.
\]
Thus $C_p$ decays  polynomially as $p\to \infty$.

Let  \(f\in C(S^{n-1})\) be even. To show that
\[
S_p f\to f \text{  as } p\to \infty,
\text{  uniformly on }S^{n-1},
\] we proceed as follows.
Fix \(\eta>0\). Since \(f\) is uniformly continuous on \(S^{n-1}\), there
exists \(\delta\in(0,\sqrt{2})\) such that 
\[
|f(u)-f(v)|<\eta
\qquad\text{whenever }|u-v|<\delta.
\]
Now fix \(\theta\in S^{n-1}\), and let
\[
E_\delta(\theta)
=
\{u\in S^{n-1}:|u-\theta|<\delta\}
\cup
\{u\in S^{n-1}:|u+\theta|<\delta\}.
\]
Since $f$ is even, we have
\[
|f(u)-f(\theta)|<\eta
\qquad\text{for every }u\in E_\delta(\theta).
\]
If \(u\notin E_\delta(\theta)\),  then
\[
|\langle u,\theta\rangle|
\le q:=1-\frac{\delta^2}{2}<1.
\]

Therefore,
\[
\begin{aligned}
	|S_p f(\theta)-f(\theta)|
	&\le
	\frac{1}{C_p}
	\int_{E_\delta(\theta)}
	|\langle\theta,u\rangle|^p
	|f(u)-f(\theta)|\,du
+
	\frac{1}{C_p}
	\int_{S^{n-1}\setminus E_\delta(\theta)}
	|\langle\theta,u\rangle|^p
	|f(u)-f(\theta)|\,du
	\\
	&\le
	\eta
	+
	\frac{2\|f\|_\infty |S^{n-1}|q^p}{C_p}.
\end{aligned}
\]

The right-hand side is independent of \(\theta\). Since
$q^p$ decays exponentially while $C_p$ decays only polynomially, there exists \(p_0\) such that
\[
\frac{2\|f\|_\infty |S^{n-1}|q^p}{C_p}<\eta
\]
for every \(p\ge p_0\). Therefore, for every \(p\ge p_0\),
\[
|S_p f(\theta)-f(\theta)|<2\eta
\qquad\text{for all }\theta\in S^{n-1}.
\]
Hence
\[
\|S_p f-f\|_\infty\longrightarrow0.
\]

\medskip
\noindent
\textbf{Step 3.}  
 We claim that there exists \(p_\ast>0\),  such that for every
 \(p\ge p_\ast\) and every \(f\in W\),
 \begin{equation}\label{eq:uniform-lower-bound}
 	\|S_p f\|_\infty
 	\ge
 	\frac12\|f\|_\infty.
 \end{equation}

To prove \eqref{eq:uniform-lower-bound}, choose a basis
\[
\varphi_1,\dots,\varphi_d
\]
of \(W\). For
\[
f=\sum_{j=1}^d \alpha_j\varphi_j,
\]
define the coefficient norm
\[
\|f\|_{\mathrm{coef}}
=
\sum_{j=1}^d|\alpha_j|.
\]
Since \(W\) is finite-dimensional, the coefficient norm and the uniform norm
are equivalent. Hence there exists a constant \(M>0\), depending only on
the chosen basis of \(W\), such that
\[
\sum_{j=1}^d|\alpha_j|
\le
M\left\|
\sum_{j=1}^d\alpha_j\varphi_j
\right\|_\infty
=
M\|f\|_\infty
\]
for every \(f\in W\).

By linearity of \(S_p\),
\[
S_pf-f
=
\sum_{j=1}^d
\alpha_j\bigl(S_p\varphi_j-\varphi_j\bigr).
\]
Therefore,
\[
\begin{aligned}
	\|S_pf-f\|_\infty
	&\le
	\sum_{j=1}^d
	|\alpha_j|\,
	\|S_p\varphi_j-\varphi_j\|_\infty
	\le
	\left(\sum_{j=1}^d|\alpha_j|\right)
	\max_{1\le j\le d}
	\|S_p\varphi_j-\varphi_j\|_\infty
	\\
	&\le
	M\|f\|_\infty
	\max_{1\le j\le d}
	\|S_p\varphi_j-\varphi_j\|_\infty.
\end{aligned}
\]
For every \(j=1,\dots,d\),
\[
\|S_p\varphi_j-\varphi_j\|_\infty\longrightarrow0
\qquad\text{as }p\to\infty.
\]
Consequently, there exists \(p_\ast>0\),  such that
\[
M
\max_{1\le j\le d}
\|S_p\varphi_j-\varphi_j\|_\infty
\le \frac12
\]
for every \(p\ge p_\ast\). Hence
\[
\|S_pf-f\|_\infty
\le
\frac12\|f\|_\infty
\]
for every \(p\ge p_\ast\) and every \(f\in W\).

Finally, by the triangle inequality,
\[
\|S_pf\|_\infty
\ge
\|f\|_\infty-\|S_pf-f\|_\infty
\ge
\frac12\|f\|_\infty.
\]
This proves \eqref{eq:uniform-lower-bound}.

\medskip
\noindent
\textbf{Step 4.} We apply a limiting procedure as $k\to\infty$ to \eqref{norm-vs-S_p} to show that $\rho_K(\theta)=a$ on some open subset of the sphere. 

For  $k\ge n$, define
\[
 F_k(\theta)=\|\theta\|_K^{p_k}+\|-\theta\|_K^{p_k},
\]
where $p_k=-n+k+1$.

By \eqref{norm-vs-S_p}, for all even $k\ge n$, 
\begin{equation}\label{eq:Fk-operator}
 F_k=C_{p_k}S_{p_k}w_k,
 \qquad w_k\in W.
\end{equation}
Let
\[
 G_k=\|F_k\|_\infty.
\]
Since $\|\theta\|_K=1/\rho_K(\theta)\le1/a$ for all $\theta$, and equality
holds in at least one direction,
\[
 \left(\frac1a\right)^{p_k}
 \le G_k\le 2\left(\frac1a\right)^{p_k}.
\]
As $p_k/k\to1$, we obtain
\begin{equation}\label{eq:Gk-root}
 G_k^{1/k}\longrightarrow\frac1a.
\end{equation}

Define
\[
 z_k=\frac{C_{p_k}w_k}{G_k}\in W.
\]
Equation \eqref{eq:Fk-operator} gives
\begin{equation}\label{eq:normalized-Fk}
 \frac{F_k}{G_k}=S_{p_k}z_k.
\end{equation}
The left-hand side has uniform norm equal to one. By
\eqref{eq:uniform-lower-bound}, $\|z_k\|_\infty\le2$ for all sufficiently
large $k$. Since $W$ is finite-dimensional, after passing to a subsequence of
even integers we may assume that
\[
 z_k\longrightarrow z
 \qquad\text{uniformly on } S^{n-1}
\]
for some $z\in W$.

The operators $S_p$ are contractions on $C(S^{n-1})$:
\[
 \|S_pf\|_\infty\le\|f\|_\infty.
\]
Since $z$ is even and continuous, $S_{p_k}z\to z$ uniformly. Hence
\begin{align*}
 \|S_{p_k}z_k-z\|_\infty&\le\|S_{p_k}z_k-S_{p_k}z\|_\infty+\|S_{p_k}z-z\|_\infty\\ 
 &\le\|z_k-z\|_\infty+\|S_{p_k}z-z\|_\infty\longrightarrow0.
\end{align*}
Together with \eqref{eq:normalized-Fk}, this yields
\begin{equation}\label{eq:normalized-limit}
 \frac{F_k}{G_k}\longrightarrow z
 \qquad\text{uniformly on } S^{n-1}.
\end{equation}
Since $F_k/G_k\ge0$ and $\|F_k/G_k\|_\infty=1$, we have
$z\ge0$ and $\|z\|_\infty=1$. Thus
\[
 U_0=\{\theta\in S^{n-1}:z(\theta)>0\}
\]
is non-empty and open.

For $\theta\in U_0$, \eqref{eq:normalized-limit} gives
\[
 \left(\frac{F_k(\theta)}{G_k}\right)^{1/k}\longrightarrow1.
\]
By \eqref{eq:Gk-root},
\[
 F_k(\theta)^{1/k}\longrightarrow\frac1a.
\]
On the other hand,
\[
 \max\{\|\theta\|_K,\|-\theta\|_K\}^{p_k}
 \le F_k(\theta)
 \le2\max\{\|\theta\|_K,\|-\theta\|_K\}^{p_k},
\]
so, since $p_k/k\to1$,
\begin{equation*}\label{eq:max-equality}
 \max\{\|\theta\|_K,\|-\theta\|_K\}=\frac1a,
 \qquad \theta\in U_0.
\end{equation*}

We claim that there is an open set $U\subset S^{n-1}$, such that 
\[
\|\theta\|_K=\frac1a
\qquad\text{for every }\theta\in U.
\]
To this end, let
\[
E_+
=
\left\{
\theta\in S^{n-1}:\|\theta\|_K=\frac1a
\right\},
\qquad
E_-
=
\left\{
\theta\in S^{n-1}:\|-\theta\|_K=\frac1a
\right\}.
\]
We have proved that
\[
U_0\subset E_+\cup E_-.
\]
If \(U_0\subset E_+\), then \(\|\theta\|_K=1/a\) on the non-empty open
set \(U_0\). Otherwise,
\[
U_0\setminus E_+
=
U_0\cap\bigl(S^{n-1}\setminus E_+\bigr)
\]
is a non-empty open set, because \(E_+\) is closed. Moreover,
\[
U_0\setminus E_+\subset E_-.
\]
Consequently,
\[
\|-\theta\|_K=\frac1a
\qquad\text{for every }\theta\in U_0\setminus E_+.
\]
Replacing this set by its negative, we obtain a non-empty open set
\(U\subset S^{n-1}\) such that
\[
\|\theta\|_K=\frac1a
\qquad\text{for every }\theta\in U.
\]

\medskip
\noindent
\textbf{Step 5.} Here we show that $K$ is equal to the ball  $B=aB_2^n$. 
 
We have shown that $\rho_K(\theta) = a=\rho_B(\theta)$ for $\theta$ from some open set $U$.  Fix
$\theta_0\in U$. Since $K$ and $B$ coincide in the neighborhood of $a\theta_0$, there is $\delta>0$ such that 
\[
 A_{K,\theta_0}(t)
 =\kappa_{n-1}(a^2-t^2)^{(n-1)/2},
 \qquad a-\delta<t<a.
\]
Here, $\kappa_{n-1}= \mathrm{vol}_{n-1}(B_2^{n-1})$.

By Lemma~\ref{lem:continuation-separable}, $A_{K,\theta_0}$ is real analytic on
$(-a,a)$. The parallel section function of the ball $B$
\[
 \kappa_{n-1}(a^2-t^2)^{(n-1)/2}
\]
is also real analytic on $(-a,a)$. Therefore,
\begin{equation}\label{eq:all-section-volumes}
 A_{K,\theta_0}(t)
 =\kappa_{n-1}(a^2-t^2)^{(n-1)/2},
 \qquad -a<t<a.
\end{equation}

 By the definition of $a$, $B\subset K$. Thus, for every $|t|<a$,
\[
 B\cap\{x:\langle x,\theta_0\rangle=t\}
 \subset
 K\cap\{x:\langle x,\theta_0\rangle=t\}.
\]
By \eqref{eq:all-section-volumes}, these two   sections have the
same $(n-1)$-dimensional volume, so the inclusion is equality. Therefore
\begin{equation*}\label{eq:slab}
 K\cap\{x:|\langle x,\theta_0\rangle|<a\}
 =B\cap\{x:|\langle x,\theta_0\rangle|<a\}.
\end{equation*}
By convexity, there are no points of $K$ outside the slab $\{x:|\langle x,\theta_0\rangle|\le a\}$, and thus $K=B$.

\end{proof}

 \section{Applications to the homothety conjecture for floating bodies}\label{Section-floating}

 \begin{theorem}\label{thm:floating-body-homothety}
  	Let \(K\subset\mathbb{R}^n\), \(n\ge 3\), be an origin-symmetric infinitely
  	smooth convex body. Suppose that there exists
  	\[
  	0<\delta_0<\frac{\operatorname{vol}_n(K)}{2}
  	\]
  	such that \(K_\delta\) is homothetic to \(K\) for every
  	\[
  	\delta\in
  	\left(\delta_0,\frac{\operatorname{vol}_n(K)}{2}\right).
  	\]
  	Then \(K\) is an ellipsoid.
  \end{theorem}
  
  \begin{proof}
  	Since \(K\) is origin-symmetric, its floating body \(K_\delta\) is also
  	origin-symmetric. Therefore, any homothety between \(K\) and \(K_\delta\)
  	must be centered at the origin. Thus, for every
  	\[
  	\delta\in
  	\left(\delta_0,\frac{\operatorname{vol}_n(K)}{2}\right),
  	\]
  	there exists \(\lambda=\lambda(\delta)\in(0,1)\) such that
  	\[
  	K_\delta=\lambda K.
  	\]
  	
  	The hyperplane supporting \(K_\delta=\lambda K\) with outer normal
  	\(\theta\) is
  	\[
  	\{x\in\mathbb{R}^n:
  	\langle x,\theta\rangle=\lambda h_K(\theta)\}.
  	\]
This hyperplane cuts off from
  	\(K\) a cap of volume \(\delta\). Hence
  	\begin{equation}\label{eq:cap-volume-delta}
  		\int_{\lambda h_K(\theta)}^\infty
  		A_{K,\theta}(t)\,dt
  		=
  		\delta
  	\end{equation}
  	for every \(\theta\in S^{n-1}\).
  	
  	For a fixed \(\theta\), the function
  	\[
  	\lambda\longmapsto
  	\int_{\lambda h_K(\theta)}^\infty A_{K,\theta}(t)\,dt
  	\]
  	is continuous and strictly decreasing on \((0,1)\), with values
  	$
  	\frac{\operatorname{vol}_n(K)}{2}$
  	at $\lambda=0$, and $0$ at $\lambda=1.$
  	Consequently, instead of parametrizing the homothetic floating bodies by
  	\(\delta\), we may parametrize them by \(\lambda\). Thus there exists
  	\(\lambda_0\in(0,1)\) and a function \(f:(0,\lambda_0)\to\mathbb{R}\)
  	such that
  	\begin{equation}\label{eq:constant-cap-volume}
  		\int_{\lambda h_K(\theta)}^\infty
  		A_{K,\theta}(t)\,dt
  		=
  		f(\lambda)
  	\end{equation}
  	for every \(\theta\in S^{n-1}\) and every
  	\(\lambda\in(0,\lambda_0)\). In fact, \(f(\lambda)\) is the value of
  	\(\delta\) corresponding to the homothety \(K_\delta=\lambda K\). 
  	
 Fixing \(\theta\), the left-hand side of \eqref{eq:constant-cap-volume} defines a
\(C^\infty\) extension of \(f\) to a neighborhood of zero.  Differentiating both sides of \eqref{eq:constant-cap-volume},
 we obtain
  \begin{equation}\label{eq:f-first-derivative}
  	f'(\lambda)
  	=
  	-h_K(\theta)
  	A_{K,\theta}\bigl(\lambda h_K(\theta)\bigr).
  \end{equation}
  More generally, for every integer \(k\ge0\),
  \begin{equation}\label{eq:f-higher-derivatives}
  	f^{(k+1)}(\lambda)
  	=
  	-h_K^{k+1}(\theta)
  	A_{K,\theta}^{(k)}
  	\bigl(\lambda h_K(\theta)\bigr).
  \end{equation}
  	In particular,
  	\begin{equation}\label{eq:derivatives-at-zero}
  		A_{K,\theta}^{(k)}(0)
  		=
  		-\frac{f^{(k+1)}(0)}{h_K^{k+1}(\theta)}.
  	\end{equation}
  	
  	\medskip
  	\noindent
  	\emph{Suppose first that \(n\) is odd.}
  	Applying Lemma~\ref{mainLemma} to the constant spherical harmonic
  	\(H\equiv1\), we see that
  	\[
  	M(t)
  	=
  	\int_{S^{n-1}}A_{K,\theta}(t)\,d\theta
  	\]
  	is a polynomial in \(t\). Consequently, there exists $N>0$ such that
  	\[
  	M^{(k)}(0)=0
  	\qquad\text{for every }k\ge N.
  	\]
  	Using \eqref{eq:derivatives-at-zero}, we obtain
  	\[
  	\begin{aligned}
  		0
  		=
  		M^{(k)}(0)
  		&=
  		\int_{S^{n-1}}A_{K,\theta}^{(k)}(0)\,d\theta\\
  		&=
  		-f^{(k+1)}(0)
  		\int_{S^{n-1}}h_K^{-k-1}(\theta)\,d\theta.
  	\end{aligned}
  	\]
  	The last integral is strictly positive. Therefore,
  	\[
  	f^{(k+1)}(0)=0
  	\qquad\text{for every }k\ge N.
  	\]
  	It follows from \eqref{eq:derivatives-at-zero} that
  	\[
  	A_{K,\theta}^{(k)}(0)=0
  	\qquad
  	\text{for every }\theta\in S^{n-1}
  	\text{ and every }k\ge N.
  	\]
  	By Theorem~\ref{thm A Poly Ellipsoid}, \(K\) is an ellipsoid.
  	
  	\medskip
  	\noindent
  	\emph{Suppose now that \(n\ge4\) is even.}
  	Taking \(\lambda=0\) in \eqref{eq:f-first-derivative}, we obtain
  	\begin{equation}\label{eq:central-section}
  		A_{K,\theta}(0)
  		=
  		-\frac{f'(0)}{h_K(\theta)}.
  	\end{equation}
  	Since \(A_{K,\theta}(0)>0\), we have \(f'(0)<0\). Moreover,
  	\[
  	h_K(\theta)=\|\theta\|_{K^\circ},
  	\]
  	and therefore
  	\[
  	A_{K,\theta}(0)
  	=
  	-f'(0)\|\theta\|_{K^\circ}^{-1}.
  	\]
  	Recalling that
  	\[
  	\rho_{IK}(\theta)=A_{K,\theta}(0),
  	\]
  	we conclude that
  	\begin{equation}\label{eq:IK-polar}
  		IK=c_1K^\circ
  	\end{equation}
  	for some \(c_1>0\).
  	
  	Taking \(k=n-2\) in \eqref{eq:derivatives-at-zero} gives
  	\begin{equation}\label{eq:high-section-derivative}
  		A_{K,\theta}^{(n-2)}(0)
  		=
  		-f^{(n-1)}(0)h_K^{-n+1}(\theta)
  		=
  		-f^{(n-1)}(0)\|\theta\|_{K^\circ}^{-n+1}.
  	\end{equation}
  	By \eqref{GKS},
  	\[
  	\left(\|\cdot\|_K^{-1}\right)^\wedge(\theta)
  	=
  	(-1)^{\frac{n-2}{2}}\pi
  	A_{K,\theta}^{(n-2)}(0).
  	\]
  	Combining this with \eqref{eq:high-section-derivative}, we obtain
  	\begin{equation}\label{eq:fourier-second-relation}
  		\left(\|\cdot\|_K^{-1}\right)^\wedge(\theta)
  		=
  		c_2\|\theta\|_{K^\circ}^{-n+1},
  	\end{equation}
  	where
  	\[
  	c_2=(-1)^{n/2}\pi f^{(n-1)}(0).
  	\]
  	
  	Applying the Fourier transform to
  	\eqref{eq:fourier-second-relation} and using that both sides are even functions of $\theta$, we get
  	\[
  	(2\pi)^n\|\theta\|_K^{-1}
  	=
  	c_2
  	\left(\|\cdot\|_{K^\circ}^{-n+1}\right)^\wedge(\theta).
  	\]
  	Applying \eqref{GKS} to \(K^\circ\), we obtain
  	\[
  	\left(\|\cdot\|_{K^\circ}^{-n+1}\right)^\wedge(\theta)
  	=
  	\pi(n-1)A_{K^\circ,\theta}(0).
  	\]
  	Hence
  	\[
  	\|\theta\|_K^{-1}
  	=
  	\frac{c_2\pi(n-1)}{(2\pi)^n}
  	A_{K^\circ,\theta}(0).
  	\]
  	Both sides are positive, so \(c_2>0\). Since
  	\[
  	\rho_K(\theta)=\|\theta\|_K^{-1},
  	\qquad
  	\rho_{I(K^\circ)}(\theta)=A_{K^\circ,\theta}(0),
  	\]
  	there exists \(c_3>0\) such that
  	\begin{equation}\label{eq:polar-intersection}
  		I(K^\circ)=c_3K.
  	\end{equation}
  	
  	Combining \eqref{eq:IK-polar} and
  	\eqref{eq:polar-intersection}, and using the property
  	\[
  	I(cL)=c^{n-1}I(L),
  	\]
  	we obtain
  	\[
  	I^2K
  	=
  	I(c_1K^\circ)
  	=
  	c_1^{n-1}I(K^\circ)
  	=
  	cK
  	\]
  	for some \(c>0\). By \cite[Theorem~1.1]{msy}, \(K\) is an ellipsoid.
  \end{proof}

We will now treat the case of $\mathbb R^2$, where we can prove a stronger statement.

\begin{theorem}
	Let $K\subset\mathbb R^2$ be a convex body with
	$0\in\operatorname{int}K$. Assume that
	\[
	K_\delta=\lambda(\delta)K, \qquad \lambda(\delta)>0,
	\]
	for every $\delta$ in a non-empty open interval
	$(\delta_1,\delta_2)$. Then $K$ is an ellipse.
\end{theorem}

\begin{proof}
	By \cite{WY}, \(K\) is strictly convex and \(\partial K\) is of class
	\(C^1\). For \(\theta\in S^1\), let \(x(\theta)\) be the unique point of
	\(\partial K\) with outer normal \(\theta\), and let
	\(m(t,\theta)\) denote the midpoint of the chord
	$$
	K\cap\{x:\langle x,\theta\rangle=t\}.
	$$
	
	For \(\delta\in(\delta_1,\delta_2)\), the supporting line of
	\(K_\delta=\lambda(\delta)K\) with outer normal \(\theta\) is	
	$$
	\{x:\langle x,\theta\rangle=\lambda(\delta)h_K(\theta)\}.
	$$
	
	By the Dupin property of the floating body, the point of contact is the
	midpoint of the corresponding chord of \(K\). Hence	
	\begin{equation}\label{eqn1}
	m\bigl(\lambda(\delta)h_K(\theta),\theta\bigr)
	=
	\lambda(\delta)x(\theta).
	\end{equation}
	Since the area cut off by
	\(\{x:\langle x,\theta\rangle=t\}\) is continuous and strictly decreasing
	in \(t\), the function \(\lambda(\delta)\) is continuous and strictly
	decreasing. Thus	
	$$
	I:=\lambda((\delta_1,\delta_2))
	$$	
	is an open interval contained in \((0,1)\). From \eqref{eqn1}, for every
	\(r\in I\),	
	\begin{equation}\label{eqn2}
	m(rh_K(\theta),\theta)=r\,x(\theta).
	\end{equation}
	Thus, for each fixed \(\theta\), the midpoints of all chords in this open
	family lie on the line \(\mathbb R x(\theta)\).
	
We now use the local form of Soltan's result
\cite[Theorem~2]{Soltan2011}. His planar argument shows that if, in an
open family of chords near a fixed chord, the midpoints of the parallel
chords in each fixed direction lie on a line, then a neighborhood of
each endpoint of the fixed chord is contained in a conic. Therefore,
since the family of chords in \eqref{eqn2} is open, \eqref{eqn2}
implies that a neighborhood of each endpoint of every chord
\[
K\cap\{x:\langle x,\theta\rangle=rh_K(\theta)\},
\qquad r\in I,
\]
is contained in a conic.
	
	Observe that every point of \(\partial K\) is an
	endpoint of such a chord.
	Thus every point of \(\partial K\) has a neighborhood contained in a
	conic. Two conics containing a common nontrivial arc coincide, so by the
	connectedness of \(\partial K\) all these local conics are one and the
	same. Since \(\partial K\) is compact and strictly convex, this conic is
	an ellipse. Therefore \(K\) is an ellipse.
\end{proof}

 \section{Applications to the problem of Barker and Larman}\label{Section:BL}

The following theorem is a particular case of \cite[Theorem 7.2.18]{g}
\begin{theorem}
\label{thm p-chordal K=L}
    Let $K$ and $L$ be convex bodies in $\mathbb{R}^n$. If $p_1,p_2 \in \operatorname{int} K \cap \operatorname{int} L$ are distinct points such that 
   	\[
	A_{K-p_i,\theta}(0)=A_{L-p_i,\theta}(0),
	\qquad \theta\in S^{n-1},
	\]
    for  $i=1,2$, then $K=L$.
\end{theorem}
 
In \cite{BL}, Barker and Larman posed the following question. Let
\(t>0\) be fixed, and let \(K\) and \(L\) be convex bodies in
\(\mathbb{R}^n\) containing \(tB_2^n\) in their interiors. Does
\[
A_{K,\theta}(t)=A_{L,\theta}(t),
\qquad \theta\in S^{n-1},
\]
imply that \(K=L\)?

This problem remains open. Barker and Larman proved that, when \(n\)
is odd, equality of the parallel section functions for infinitely many
values of \(t\) implies that \(K=L\). They also asked whether the
corresponding result holds in even dimensions. The following theorem
answers this question affirmatively.
 
 \begin{theorem}Let \(K\) and \(L\) be convex bodies in \(\mathbb{R}^n\), \(n\geq 2\),
	containing the origin in their interiors, and set
	\[
	a=
	\min_{u\in S^{n-1}}
	\min\{\rho_K(u),\rho_L(u)\}.
	\]
	Let \(T\subset(-a,a)\) be an infinite set having an accumulation point
	in \((-a,a)\). Suppose that, for every \(t\in T\),
	\[
	A_{K,\theta}(t)=A_{L,\theta}(t),
	\qquad \theta\in S^{n-1}.
	\]
	Then \(K=L\).
 	\end{theorem}

\begin{proof}
	Let $H$ be a spherical harmonic. By
	Lemma~\ref{mainLemma}, the functions
	\[
	t\longmapsto
	\int_{S^{n-1}}A_{K,\theta}(t)H(\theta)\,d\theta
	\]
	and
	\[
	t\longmapsto
	\int_{S^{n-1}}A_{L,\theta}(t)H(\theta)\,d\theta
	\]
	are real analytic on \((-a,a)\).
	
	By the hypothesis, these two functions agree for every \(t\in T\).
	Since \(T\) has an accumulation point in \((-a,a)\), the identity
	theorem for real-analytic functions gives
	\[
	\int_{S^{n-1}}A_{K,\theta}(t)H(\theta)\,d\theta
	=
	\int_{S^{n-1}}A_{L,\theta}(t)H(\theta)\,d\theta
	\]
	for every \(t\in(-a,a)\).
	
	Thus, for every fixed \(t\in(-a,a)\), the function
	\[
	\theta\longmapsto A_{K,\theta}(t)-A_{L,\theta}(t)
	\]
	is orthogonal to every spherical harmonic. By the completeness of the
	spherical harmonics in \(L^2(S^{n-1})\), it follows that
	\[
	A_{K,\theta}(t)=A_{L,\theta}(t)
	\]
	for almost every \(\theta\in S^{n-1}\). Since the parallel section
	functions are continuous in \(\theta\), we conclude that
	\begin{equation}\label{eq:all-t-all-theta}
		A_{K,\theta}(t)=A_{L,\theta}(t),
		\qquad
		t\in(-a,a),\quad \theta\in S^{n-1}.
	\end{equation}
	
	Choose \(\xi\in S^{n-1}\) and a nonzero \(t_0\in(-a,a)\), and set
	\[
	p_1=0,
	\qquad
	p_2=t_0\xi.
	\]
	By the definition of \(a\),
	\[
	aB_2^n\subset K\cap L,
	\]
	and hence
	\[
	p_1,p_2\in\operatorname{int}K\cap\operatorname{int}L.
	\]
	
	For \(p_1=0\), equation \eqref{eq:all-t-all-theta} gives
	\[
	A_{K-p_1,\theta}(0)
	=
	A_{K,\theta}(0)
	=
	A_{L,\theta}(0)
	=
	A_{L-p_1,\theta}(0)
	\]
	for every \(\theta\in S^{n-1}\).
	
	For \(p_2=t_0\xi\), translating the section hyperplane gives
	\[
	A_{K-p_2,\theta}(0)
	=
	A_{K,\theta}(\langle p_2,\theta\rangle)
	=
	A_{K,\theta}(t_0\langle\xi,\theta\rangle).
	\]
	Since
	\[
	\left|t_0\langle\xi,\theta\rangle\right|
	\leq |t_0|<a,
	\]
	equation \eqref{eq:all-t-all-theta} implies that
$$
		A_{K-p_2,\theta}(0)
	=
		A_{K,\theta}(t_0\langle\xi,\theta\rangle) 
		=
		A_{L,\theta}(t_0\langle\xi,\theta\rangle)
		=
		A_{L-p_2,\theta}(0)
$$
	for every \(\theta\in S^{n-1}\).
	
	The hypotheses of Theorem~\ref{thm p-chordal K=L} are therefore
	satisfied for the two distinct points \(p_1\) and \(p_2\). Consequently,
	\(K=L\).
\end{proof}

\textbf{AI Disclaimer.} The authors used ChatGPT for proof checking and improving the exposition.
    
    \textbf{Acknowledgment.} This project was initiated during visits by D. Ryabogin and V. Yaskin to the University of Sydney in 2024 and 2025. They are grateful to SMRI and the University of Sydney for their hospitality and financial support.


\begin{thebibliography}{99} 

\bibitem{A}
M.~Agranovsky, \emph{On polynomially integrable domains in Euclidean spaces},
Trends in Mathematics; Complex Analysis and Dynamical Systems. New Trends and Open Problems, Birkh\"auser (2018), pp. 1--21.

\bibitem{A2}
M.~Agranovsky,
\emph{Locally polynomially integrable surfaces and finite stationary phase expansions},
J. Anal. Math. \textbf{141} (2020), no. 1, 23--47.

\bibitem{akry}
M. Agranovsky, A. Koldobsky, D. Ryabogin, V. Yaskin, \emph{An analogue of polynomially integrable bodies in even-dimensional spaces}, J. Math. Anal. Appl. \textbf{529} (2024), no. 2, Paper No. 127071, 12 pp.

\bibitem{ANRSY}
M.~Alfonseca, F.~Nazarov, D.~Ryabogin, A.~Stancu, V.~Yaskin, 
\emph{Two results on the homothety conjecture for convex bodies of flotation on the plane}, 
Rev. Un. Mat. Argentina \textbf{69} (2026), no. 1, 93--108.

\bibitem{Ar}
V.~I.~Arnold, \emph{Arnold’s problems}, 2nd edition, Springer-Verlag, Berlin (2004).

\bibitem{BarL} I.~B\'ar\'any and D.~G.~Larman, \emph{Convex bodies, economic cap coverings, random polytopes},
Mathematika, \textbf{35 (2)} (1988), 274--291.

\bibitem{BL} J.~A.~Barker, D.~G.~Larman, \emph{Determination of convex bodies by certain sets of sectional volumes}, Discrete Mathematics \textbf{241} (2001), 79--96.

\bibitem{g}
R. Gardner, \emph{Geometric Tomography}, second ed., Cambridge University Press, New York, 2006.

\bibitem{GKS}
R.~J.~Gardner, A.~Koldobsky, T.~Schlumprecht, \emph{An analytic solution to the
Busemann-Petty problem on sections of convex bodies}, Annals of Math. \textbf{149}
(1999), 691--703.



\bibitem{GYY}
P.~Goodey, V.~Yaskin,  M.~Yaskina, \emph{Fourier transforms and the Funk-Hecke theorem in convex geometry}, J. Lond. Math. Soc. (2) \textbf{80} (2009), no. 2, 388–404.

\bibitem{Gro}
H.~Groemer, \emph{Geometric Applications of Fourier Series and Spherical Harmonics}, Cambridge
University Press, New York, 1996.


\bibitem{H}
J.~Haddad, \emph{An approximate counterexample to the Barker–Larman problem in dimension 4},
preprint, arXiv:2609.10257.

\bibitem{HR}
 J.~Haddad, D.~Ryabogin, \emph{On convex bodies with constant non-central sections}, preprint, arXiv:2605.00299.



\bibitem{K-book}
A. Koldobsky, \emph{Fourier Analysis in Convex Geometry}, American Mathematical Society, Providence RI, 2005.

\bibitem{kmy}
A.~Koldobsky, A.~S.~Merkurjev, V.~Yaskin, \emph{On polynomially integrable convex bodies}, Adv. Math. \textbf{320} (2017) 876–886.


  \bibitem{L}
  {E.~Lutwak}, \textit{Intersection bodies and dual mixed volumes}, Adv. Math. {\bf 71} (1988),
  232--261.

\bibitem{msy}
E.~Milman, S.~Shabelman, A.~Yehudayoff, \emph{Fixed and periodic points of the intersection body operator}, Invent. Math. \textbf{241} (2025), no. 2, 509--558.



\bibitem{RY}
D.~Ryabogin and V.~Yaskin, \emph{Detecting symmetry in star bodies}, J. Math. Anal. Appl. \textbf{395} (2012),
509--514.

\bibitem{S} L.~A.~Santal\'o, \emph{Two characteristic properties of circles on a spherical surface} (Spanish), Math. Notae
\textbf{11} (1951), 73--78.

\bibitem{SW} C.~Sch\"utt, E.~Werner, \emph{The convex floating body}, Math. Scand. \textbf{66} (1990), 275--290.

\bibitem{SW2} C.~Sch\"utt, E.~Werner, \emph{Homothetic floating bodies}, Geom. Dedicata, \textbf{49 (3)} (1994), 335--348.


\bibitem{Soltan2011}
V.~Soltan,
\emph{Convex solids with hyperplanar midsurfaces for restricted families of chords},
Bul. Acad. \c{S}tiin\c{t}e Repub. Moldova Mat.
\textbf{2011}, no.~2 (66), 23--40.


\bibitem{St} 
A.~Stancu,  
\emph{The floating body problem},
Bull. London Math. Soc. \textbf{38} (2006), no. 5, 839--846.

\bibitem{WY} 
E.~Werner, D.~Ye,  
\emph{On the homothety conjecture}, 
Indiana Univ. Math. J. \textbf{60} (2011), no. 1, 1--20.

\bibitem{XMC} 
Ge Xiong, Yan-wei Ma, Wing-sum Cheung, \emph{Determination of convex bodies from
$\Gamma$-section functions}, J. Shanghai Univ. \textbf{12} (2008), no. 3, 200--203.

\bibitem{Y}  V.~Yaskin, 
\emph{Unique determination of convex polytopes by non-central sections}, 
Math. Ann. \textbf{349} (2011), no. 3, 647--655.

\bibitem{yz}
V.~Yaskin and B.~Zawalski,  \emph{On separably integrable symmetric convex bodies}, Adv. Math. \textbf{441} (2024), Paper No. 109527, 24 pp.

\bibitem{YZh}
V.~Yaskin, N.~Zhang,  
\emph{Non-central sections of convex bodies},
Israel J. Math. \textbf{220} (2017), no. 2, 763--790.


\end{thebibliography}
\end{document}